\documentclass[a4paper, 11pt, reqno]{amsart}

\usepackage{setspace}\usepackage{mathrsfs}
\usepackage[utf8]{inputenc}
\usepackage{graphicx}\usepackage{constants}
\usepackage{amsmath}
\usepackage{amssymb,amsthm}
\usepackage{cases}
\usepackage{esint}
\usepackage[a4paper,margin=3cm]{geometry}
\usepackage[citecolor=blue, linkcolor=blue,colorlinks]{hyperref}

\graphicspath{ {images/} }
\newtheorem{theorem}{Theorem}
\newtheorem{lemma}{Lemma}

\newtheorem{prop}{Proposition}

\theoremstyle{definition}
\newtheorem{rk}{Remark}

\def\R{\mathbb{R}}

\def\half{{\textstyle\frac{1}{2}}}

\def\a{\alpha}
\def\e{\epsilon} 

\def\p{\partial}

\font\smathbold=msbm7\font\ssmathbold=msbm7 at 5.5pt
\def\sR{{\hbox{\smathbold\char82}}}\def\ssR{{\hbox{\ssmathbold\char82}}}

\def\ni{\noindent}

\def\ba{\begin{aligned}}
\def\ea{\end{aligned}}

 \def\d{\delta}

\def\be{\begin{equation}}\def\ee{\end{equation}}\def\bc{\begin{cases}}\def\ec{\end{cases}}
\def\inj{{\rm {inj}}}\def\N{{\mathbb N}}
\def\qed{\rightline{\setlength{\fboxsep}{0pt}\setlength{\fboxrule}{0.2pt}\fbox{\rule[0pt]{0pt}{1.3ex}\rule[0pt]{1.3ex}{0pt}}}}

\def\ric{{\rm{Ric } }}

\def\llceil{\left\lceil}\def\rrceil{\right\rceil}
\newcommand\fH[1]{\sbox0{#1}\dimen0=\ht0 \advance\dimen0 -1ex
  \sbox2{\'{}}\sbox2{\raise\dimen0\box2}%
  {\ooalign{\hidewidth\kern.1em\copy2\kern-.5\wd2\box2\hidewidth\cr\box0\crcr}}}

\font\twelvemi=cmmi12 at 12pt\font\elevenmi=cmmi11 at 9 pt\font\fivemi=cmmi5 at 6 pt\font\twelvemimi=cmmi12 at 12pt
\renewcommand{\chi}{\raisebox{.13\baselineskip}{\hbox{\twelvemi\char31}}}
\newcommand{\gam}{{\raisebox{.08\baselineskip}{\hbox{\twelvemi\char13}}}}
\newcommand{\sgam}{{\raisebox{.0\baselineskip}{\hbox{\elevenmi\char13}}}}
\newcommand{\sm}{{\hbox{\twelvemimi\char27}}_{\!\hbox{\fivemi\char 77}}}

\renewcommand{\gamma}{\gam}
\def\vol{{\rm {Vol }}}

\vskip 0.1in\def\grn_#1{{G_#1^{\sR^{\hskip-.007in n}}}}\def\sgrn_#1{{G_#1^{\ssR^{\hskip-.007in n}}}}
\def\per{{\rm {Per }}}

\def\nn{{\frac{n-1}{n}}}\def\nnn{{\frac n{n-1}}}\def\NN{{\mathcal N}}
\def\PS{P\'olya-Szeg\H o }\def\gp{{g_{p}^{}}}
\def\vt{\tilde v}\def\rh{r_{\hskip-.25em\scriptscriptstyle H} }
\def\diam{{\rm diam }}\def\dh{\delta_{\hskip-.1em\scriptscriptstyle H} }
\def\im{I_{\hskip-.1em\scriptscriptstyle M} }\def\ipa{I_p^\a}
\def\nk{{n/k}}\def\nnk{{\frac{n}{n-k}}}\def\knk{{\frac{k}{n-k}}}
\title[Sharp Sobolev inequality] {Sharp Sobolev and Moser-Trudinger inequalities on noncompact Riemannian manifolds with Ricci curvature bounded below}
\author[ C. Morpurgo,  L. Qin] { Carlo Morpurgo,   Liuyu Qin}
 \thanks{The third author was supported by the National Natural Science Foundation of China (12201197)}
\begin{document}

\begin{abstract} We establish Sobolev and Moser-Trudinger inequalities with best constants on noncompact Riemannian manifolds with Ricci curvature bounded below, and positive injectivity radius.

\end{abstract}
\numberwithin{equation}{section}
\maketitle
\tableofcontents


\section{Introduction}\label{s1}

Given   a  smooth, connected, complete $n-$dimensional  Riemannian manifold $(M,g)$, with $n\ge2$, let $\a>0 $ and  
\be 1\le p<n,\qquad q=\frac {np}{n-p}.\label {par}\ee 
We say that the (sharp) $\ipa$ Sobolev inequality holds if there exists $B>0$ such that 
\be \|u\|_q^\a\le K(n,p)^\a\|\nabla u\|_p^\a+B\|u\|_p^\a,\qquad u\in W^{1,p}(M),\tag{$I_p^\alpha$}\label{ipa}\ee
	 where $K(n,p)$ denotes the best  constant in the Sobolev embedding $W^{1,p}(\R^n)\hookrightarrow L^q(\R^n)$, given by the well-known inequality  due to Aubin \cite{aubin} and Talenti \cite {talenti} (for the explicit value of $K(n,p)$ see \eqref{knp} below). The  constant $K(n,p)^\a$ in \eqref{ipa} is optimal, that is, it cannot be replaced by a smaller constant. For a proof of this  general fact, valid on any Riemannian manifold, see Proposition 4.2 in \cite{hebey-courant} (where it was proved for $\a=1$, but the proof also works for any $\a>0$). It's easy to check that if $I_p^\a$ holds for some $\alpha>0,$ then $I_p^\beta$ holds if $0<\beta<\alpha.$ 

In the recent paper \cite{mnq} it  is proven that $I_p^1$ holds under the stronger assumption 
\be |\ric|\le K,\qquad \inj(M)>0\label{riccistrong}.\ee 
 Previous results for $\alpha=1$  were known  for $1<p<n$, under even stronger assumptions on the full curvature tensor \cite{aubin-li}, and for $p=1$ when $M$ has constant sectional curvature \cite{aubin-jdg}.  For a more in-depth discussion on the validity of \eqref{ipa},  we refer the reader to \cite{mnq} and references therein, with particular emphasis on \cite{hebey-courant} and \cite{druet-hebey}, where Hebey's  celebrated AB program is discussed extensively. 

We highlight here the paper by Hebey-Vaugon \cite{hebey-vaugon}, where the authors derived $I_2^2$ under global bounds on the full curvature tensor and its covariant derivative, the first sharp result on noncompact manifolds. In the same paper,  the authors put forward the idea that $I_2^2$ might be valid under less restrictive conditions on the geometry, in particular under global bounds on the Ricci curvature (see \cite{hebey-vaugon}, Remark on page 272). Shortly after, in \cite{hebey-ajm} Hebey  proved an almost sharp version of such a result. In particular, he proved that if $M$ is a smooth, complete Riemannian manifold with dimension $n\ge 3$, and such that 
\be  \ric\ge K,\qquad \inj(M)>0\label{ricciweak},\ee
then, for $0<\alpha\le p$ and for any $\e>0$ there is $B=B(\e,p)$ such that
\be \|u\|_q^\a\le (K(n,p)^\a+\e)\|\nabla u\|_p^\a+B\|u\|_p^\a,\qquad u\in W^{1,p}(M)\label{weakipa}\ee
(in \cite{hebey-ajm} estimate \eqref{weakipa} was proved for $\alpha=p$,  which directly implies \eqref{weakipa} in  the case  $0<\alpha<p$).

To our knowledge, there are no results regarding the validity of the above result when $\e=0$, i.e. the validity of \eqref{ipa} under 
\eqref{ricciweak}. The first main result of this paper is the following:

\begin{theorem}\label{main1} On a smooth, connected, complete, noncompact, $n-$dimensional  Riemannian  manifold without boundary $(M,g)$, suppose that 
\be \ric \ge K,\qquad \inj(M)>0,\label{ric}\ee
for some $K\in\R.$
 Then,  for $0<\alpha<1$ and $1\le p<n$ the $I_p^\a$ Sobolev inequality holds.

\end{theorem}

As we mentioned above, $I_p^1$ was obtained in \cite{mnq} under the stronger assumption \eqref{riccistrong}. It does not seem possible to obtain $I_p^1$ under \eqref{ric}  using the methods of our paper. One of the key steps in the proof (see below for some more details) is to establish a local version of $I_1^\a$, for any  fixed $\alpha\in (0,1)$, via a delicate blowup argument in a suitable  $C^{0,\alpha}$ harmonic chart. The existence of such chart is guaranteed by results of Anderson-Cheeger \cite{ac}, under \eqref{ric}. Likewise, a similar step was taken in \cite{mnq} for the local $I_1^1$, where the blowup this time  was inside a suitable $C^1$ harmonic chart, the existence of which is guaranteed (under \eqref{riccistrong}) by a classical result of Anderson \cite{and}.  There appears to be  a direct correlation between the exponent $\alpha$  of the local  $I_1^\a$ and the order of  local, uniform $C^{k,\alpha}$ control of the  metric in harmonic coordinates, at least when $\alpha\le 1.$
It would be interesting, and probably quite hard,  to establish whether or not $I_p^1$ is valid under \eqref{ric}.

The case $p=n$ is  the so-called borderline case, and the corresponding Sobolev embedding is given in terms of a Moser-Trudinger inequality on $W^{1,n}(M)$. In this paper we will also treat the borderline Sobolev space  $W^{2,\frac n2}(M)$.

For $k\in\N$ define
\be\nabla^k u=\bc \nabla(-\Delta)^{\frac{k-1}2} u & {\text{ if }} k\; {\text{is odd}}\\ (-\Delta)^{\frac k2} u & {\text{ if }} k\; {\text{is even}},\ec\ee
and for any $N\in\N\cup\{0\}$ define the regularized exponential function as
\be \exp_N(t)=e^t-\sum_{j=0}^N \frac { t^j} {j!}.\ee
We will say that the sharp Moser-Trudinger $MT_k$ inequality holds  ($1\le k<n$), if for any $\kappa>0$ there is $C>0$ such that for all $u\in W^{k,\frac n k}(M)$ with 
\be \kappa\|u\|_{n/k}^{n/k}+\|\nabla^k u\|_{n/k}^{n/k}\le1\label{ruf}\ee
we have 
\be  \int_M \exp_{\llceil\frac{n-k}{k}-1\rrceil}{\left(\gam_{n,k}|u|^{\frac{n}{n-k}}\right)}d\mu\le C,\tag{$MT_k$} \label{MT}\ee 
where $\gam_{n,k}$ denotes the sharp constant in the classical Moser-Trudinger inequality on bounded open sets of  $\R^n$, derived first by Moser in the case $k=1$ \cite{moser}, and later extended  to all $k\in \N$ by Adams \cite{adams}.

In particular, we have
\be \gam_{n,1}=B_n^{\frac1{n-1}} n^{\frac n{n-1}},\qquad \gam_{n,2}=B_n^{\frac2{n-2}} \big(n(n-2)\big)^{\frac n{n-2}},\label{constants}\ee
where $B_n$ denotes the volume of the unit ball on $\R^n$. The constant $\gam_{n,k}$ in \eqref{MT} is optimal, that is, it cannot be replaced by a larger number. This is a standard fact, valid on any Riemannian manifold, and can be proved using the usual ``Adams-Moser sequence'', supported in a small neighborhood of a fixed point (see e.g. \cite{f} for the construction on a compact manifold, or \cite{y}  proof of Theorem 2.3, or \cite{fmq} proof of sharpness statement of Theorem 2).

It is also  easy to see (see \eqref{expreg} below) that \eqref{MT} under \eqref{ruf}  is equivalent to the inequality
\be\int_{|u|\ge 1} e^{\sgam_{n,k}|u|^{\frac{n}{n-k}}}d\mu\le C.\ee

On $\R^n$, inequality \eqref{MT} under \eqref{ruf} was derived   in \cite{ruf},\cite{lr}, for $k=1$,  in \cite{ll1} for $k=2$ and in \cite{fm} for any $k$.   When $k=1$ the same result was obtained in \cite{fmq} when the sectional curvature satisfies $-b^2\le\ K\le 0$, unless  $n=2,3,4$, in which case it also holds when $K\le 0$ (see \cite{kr}). In \cite{fmq} inequality \eqref{MT} was proved (any $k$) on noncompact manifolds with $\ric\ge0$ and Euclidean volume growth, under additional boundedness conditions on the full curvature tensor.  For further comments and related results see \cite{fmq}, \cite{fmq-ann}, \cite{mq} and references therein.

\smallskip\medskip
The second main result of this paper is the following:

\begin{theorem}\label{main2} On a smooth, connected, complete, noncompact,  $n-$dimensional Riemannian  manifold without boundary $(M,g)$, suppose that 
\be \ric \ge K,\qquad \inj(M)>0,\label{ric1}\ee
for some $K\in\R.$
 Then, the sharp $MT_k$ Moser-Trudinger inequality holds, for $k=1, n\ge2$ and for $k=2, n\ge3.$

\end{theorem}

\smallskip
\begin{rk} \label{yang} Theorem \ref{main2} for the case $k=1$ in the so-called subcritical case (i.e. when $\gamma_{n,1}$ is replaced by any $\gamma<\gamma_{n,1}$), and under \eqref{ruf} for some $\kappa>0$, was proved by Yang \cite{y}, Theorem 2.3.
\end{rk}

\begin{rk}\label{lilu-rem}
 Theorem \ref{main2} for the case $k=1$ was stated in \cite{lilu1}, Theorem 1.3.  Adjustments on the proof have been announced in \cite{lilu2}.
\end{rk}

\bigskip
The method of proof of Theorems \ref{main1} and  \ref{main2}, follows the same scheme as the one used in \cite{mnq} (written here in a different order): 

\medskip

\ni \underline{Step 1}: establish a local $I_1^\alpha$ inequality, that is 
\be
\|u\|_{\frac n{n-1}}^\a\le K(n,1)^\a\|\nabla u\|_1^\a+B\|u\|_1^\a\ee
for all $u\in W_0^{1,1}(B(x,r_0))$, all $x\in M$, and suitable $B,r_0>0$ (independent of $x$.
We note here that  
\be K(n,1)=\frac{B_n^{-\frac1n}}n.\label{kn1}\ee
\smallskip

\ni \underline{Step 2}: use the local $I_1^\alpha$ inequality to establish the sharp lower bound
\be I_0(v)\ge  n B_n^{\frac1n}v^{\frac {n-1}n}-Bv^{\frac{n-1+\a}n},\qquad 0<v\le v_0.\label{lb}\ee

\smallskip
\ni \underline{Step 3}:  use \eqref{lb} to prove the inequalities of Theorems \ref{main1} and \ref {main2} for functions supported on open sets with small volumes.

\smallskip

\ni \underline{Step 4}:  prove that if the inequalities of  Theorems \ref{main1} and \ref {main2} are valid  for small volumes  then they are valid globally.

\smallskip
               
We would like to highlight the major differences between this paper and \cite{mnq}.  First, recall that in \cite{mnq} the case $I_p^1$ was treated under the hypothesis $|\ric|\le K$ and $\inj(M)>~0$, and the Moser-Trudinger inequalities were not considered.

 Regarding Theorem \ref{main1}, one of the major difficulties we encountered in  going from $\alpha=1$ to $\alpha<1$ was in Step 1, that is establishing  the local $I_1^\a$ inequality under \eqref{ric}.  The proof given in \cite{mnq} was adapted from the one in \cite{druet-pams}, \cite{noa-imrn}, \cite{noa-arxiv} and  relied on a sophisticated analysis of a suitable blowup family  $\{v_p\}$ associated to the functional 
  \be\frac{\|\nabla u\|_p^\a+B \|u\|_p^\a}{\|u\|_q^\a}\ee
with $\alpha=1$.      Under the hypothesis $|\ric|\le K,\; \inj(M)>0$ the blowup was constructed in  $C^1$ harmonic coordinates
as opposed to the usual geodesic normal coordinates,  with good uniform control on the components of the metric tensor and their first partials (due to a well-known result by Anderson \cite{and}).
One key technical result was the so-called ``strong pointwise bound'' on the blowup family  $v_p$. In a $ C^1$ harmonic chart centered at a blowup point, such an inequality reads as
\be v_p(\xi)\le C |\xi|^{-\frac{n-p-b}{p-1}},\qquad\quad \xi\in\R^n,\;|\xi|\ge 1\label{strong}\ee
         where $b>0$ is small enough, $C$ is independent of $p$ and $p$ is close enough to 1  (see Lemma~2 in \cite{mnq}).
         This decay bound at infinity allowed us, and the aforementioned authors, to prove various integral estimates  on $v_p^q$,  $|\nabla v_p|^p$, and their  moments. Estimate \eqref{strong} was a crucial tool also in \cite{aubin-li} (Prop. 3.1) and in several other papers.

  The proof of \eqref{strong} relies on a comparison theorem  for operators of type $u\to -\Delta_{p,g}u+a |u|^{p-2}u$, where $\Delta_{p,g}$ denotes the $p$-Laplacian in a given  metric $g$, and where $a$ is some bounded function (see \cite{aubin-li} Lemma 3.4). To apply such a result, a uniform control of the first partials of the metric $g$ was needed - and possible - in the $C^1$ harmonic chart. This is the only step where uniform control of the partials was needed (see Proof of Lemma 2 and  Remark 5 in \cite {mnq}).
  
   Under $\ric\ge K$ and $\inj(M)>0$ one can only guarantee the existence of  harmonic charts with uniform $C^{0,\a}$ control on the components of the metric tensor, by results of Anderson-Cheeger \cite{ac}, so we were not able to obtain the strong decay bound \eqref{strong} in this setting. In this paper  we bypass this difficulty, and  complete the proof of Step 1, by  relying  only on a ``weak pointwise bound'', of type 
   \be  v_p(\xi)\le C |\xi|^{-\frac np +1},\qquad\quad \xi\in\R^n,\;|\xi|\ge 1\label{weak}\ee
see \eqref{decay}, Lemma \ref{lemma}, and Remark \ref{weak-strong}.

\smallskip

The proof of Step 2 is accomplished in the same way as in \cite{mnq}: given  $v>0$, if an isoperimetric region $\Omega$  of volume $v$ exists in $M$, then, assuming \eqref{ric}, its diameter is controlled by $v$ so for small enough $v$, so estimate \eqref{lb} follows  immediately, by  essentially applying $I_1^\a$ to $\chi_\Omega.$  If such a region does not exist in $M$, then it will still exist in a $C^{0,\a}$ pointed limit manifolds in the $C^{0,\a}$ topology, and the previous argument can still be carried on.
 
 \smallskip
 
 To handle Step 3, assuming \eqref{lb}, to derive $I_p^\a$ for small volumes we use a generalized form of the  P\'olya-Szeg\H o inequality \eqref{ps}, used also in \cite{mnq},  and an extension of the argument given in \cite{mnq}, proof of Theorem 3.
 To derive instead \eqref{MT} for small volumes, for $k=1$ we use \eqref{lb} to obtain a  representation formula for the rearrangement  of $u$, in terms of its gradient, combined with \eqref{ps} and a 1-dimensional Adams inequality
 (a similar argument was used by Cianchi in a different context on $\R^n$  \cite{ci}).
 
  To obtain \eqref{MT} for $k=2$ we will use instead a more general  form of the classical Talenti comparison theorem, given in \eqref{estlap}, to represent the rearrangement of $u$ in terms of its Laplacian and the normalized isoperimetric profile, and then use a 1-dimensional Adams inequality. We will in fact obtain a Moser-Trudinger inequality for $W_0^{1,2}(\Omega)\cap W^{2,\frac n2}(\Omega)$ (for $\vol(\Omega)$ small enough) which includes the Sobolev space with Navier boundary conditions $W_\NN^{2,\frac n2}(\Omega)=W_0^{1,\frac n2}(\Omega)\cap W^{2,\frac n2}(\Omega)$ (and hence $W_0^{2,\frac n2}(\Omega)$), when $n>3$, and are equal to such spaces when $n=3$,  if $\partial \Omega$ is smooth enough. We note that Moser-Trudinger inequalities for general $W_\NN^{k,\frac nk}(\Omega)$ on  $\R^n$ (for $\partial \Omega$ smooth enough) were first discovered by Tarsi in \cite{tarsi}, and they were also based on the original Talenti's comparison theorem. 
 
 \smallskip
 Step 4 for $I_p^\a$ requires only a small extension of the proof of Theorem 2 in \cite{mnq}, from $\alpha=1$ to $0<\alpha<1$. Regarding \eqref{MT}, Step 4 for $k=1$ will be proved by a truncation argument similar to that used in \cite{fmq-ann}  in the context of Riemannian manifolds, which was itself a variation of similar arguments used in other papers in different settings, notably \cite{bm}, \cite{ll1}, \cite{ll2}, \cite{ruf}.  For $k=2$ the truncation argument still stands, if one is working with the spaces $W_0^{1,2} (\Omega)\cap W^{2,\frac n 2}(\Omega)$. This idea was first exploited by Lam-Lu in \cite{ll1} on $\R^n$, where the authors proved \eqref{MT} for $k=2$  on $\R^n$ by   truncation over  suitable level sets, reducing  to the case of Tarsi's Moser-Trudinger inequality for  $W_\NN^{2,\frac n2}(\Omega)$.

\vskip2em


\section{Preliminaries}\label{s2}

\bigskip
From now on we will assume that  $(M,g)$ is a smooth,  connected,  complete, noncompact, $n-$dimensional Riemannian manifold, without boundary and with metric tensor $g$. The geodesic distance between two points $x,y\in M$ is denoted by $d(x,y)$, and the open geodesic ball centered at $x$ and with radius $r$ will be denoted as $B(x,r)$. The manifold $(M,g)$ is equipped with the natural Riemannian measure $\mu$ which in any chart satisfies $d\mu=\sqrt{|g|} dm$, where $|g|=\det(g_{ij})$, $g=(g_{ij})$ and $dm$ is the Lebesgue measure.

The volume of a measurable set $E\subseteq M$ is defined  as $\vol(E)=\mu(E)$, and its perimeter as the total variation of $\chi_E$.

For the definitions and the basic properties of the classical Sobolev spaces $W^{k,p}$ and $W^{k,p}_0$ on Riemannian manifolds
we shall refer the reader to \cite{hebey-courant}, chapters 2 and 3. We just point out that if  $p\ge1$, then  $W^{1,p}(M)=W_0^{1,p}(M)$, the closure of $C_c^\infty(M)$ under the norm $\|u\|_p+\|\nabla u\|_p$, and if $p>1$, $\ric\ge K$, and $\inj(M)>0$, then  $W^{2,p}(M)=W_0^{2,p}(M)$, the closure of $C_c^\infty(M)$ under $\|u\|_p+\|\Delta u\|_p$ (for the latter, see \cite{mmpv}, Corollary E).

\smallskip
\ni{\bf Convention.} Throughout the paper, the dependence of various objects ($\vol, \per, \nabla,$ etc.) on the ambient metric where these objects are defined will be suppressed, unless ambiguities arise (see  for ex. Remark \ref{amb}).

\smallskip
  The \emph{decreasing rearrangement} of a measurable function $f$ on $M$ is defined as 

\be f^*(t)=\inf\big\{s\ge0:\; \mu\{x:|f(x)|>s\})\le t\big\},\qquad  t>0\label{rearr}\ee
(we will assume that $\mu(\{x:|f(x)|>s\})<\infty$  for all  $s>0$).

For $\tau>0$ we have

\be \mu( \{x:|f(x)|>f^*(\tau)\})\le \tau\le \mu(\{x:|f(x)|\ge f^*(\tau)\})\label{tau}\ee
and $f^*(\tau)>0$ if and only if $\mu(\{f\neq0\})>\tau$, in which  case we can find a measurable set $E_\tau$ such that 
\be\mu(E_\tau)=\tau,\qquad \{x:|f(x)|>f^*(\tau)\}\subseteq E_\tau\subseteq\{x:|f(x)|\ge f^*(\tau)\}.\label{set1}\ee
(see  [MQ, (3.11)] and the comments thereafter). Additionally, we have
\be \int_{E_\tau}\Phi(|f(x)|)d\mu(x)=\int_0^\tau \Phi\big(f^*(t)\big)dt\label{int1}\ee
for any nonnegative measurable $\Phi$ on $[0,\infty)$, and in \eqref{int1} we can take $E_\tau=M$ when $\tau=\infty.$

The \emph{Schwarz symmetric decreasing rearrangement} of a measurable function $f$ on $M$ is the function $f^{\#}:\R^n\to[0,\infty)$ defined as 
\be f^\#(\xi)=f^*(B_n|\xi|^n),\qquad \xi\in\R^n.\ee
\smallskip

The \emph{isoperimetric profile} of $M$ is the function $\im:\big[0,\vol
(M)\big)\to[0,\infty)$ (where $\vol(M)\in (0,\infty]$) defined as 
\be \im(v)=\inf\big\{\per(E),\; E\; {\text {measurable}},\; \vol(E)=v\big\}.\ee
In the definition of  $\im(v)$ the sets $E$ can be chosen with smooth boundary (see \cite{mfn} Thm.1).

The classical isoperimetric inequality on $\R^n$ states that for every measurable set $E$ with finite measure 
\be \frac{\per(E)}{|E|^\nn}\ge \frac{\per (B_{\sR^n}(1))}{|B_{\sR^n}(1)|^\nn}=n B_n^{\frac1n},\ee
where $B_{\sR^n}(1)$ is the unit ball of $\R^n$, hence
\be I_{\sR^n}(v)=n B_n^{\frac1n} v^\nn.\ee
The \emph{normalized isoperimetric profile} of $M$  is defined to be the function 
\be I_0(v)=\frac{\im(v)}{I_{\sR^n}(v)}=\frac{\im(v)}{n B_n^{\frac1n} v^\nn}.\ee

It turns out that if $\ric\ge K$ and the unit ball does not collapse, i.e. if there is $\delta>0$ s.t. $\vol(B(x,1))\ge \delta$ for any $x\in M$ (in particular, if $\inj(M)>0$), then $\im(v)>0$ for any $v>0$ (see \cite{afp} Remark 4.7), and the function $v\to \im(v)$ is continuous  on $\big(0,\vol(M)\big)$, in fact it is H\"olderian of order $1-\frac1n$ (\cite{mfn} Theorem 2). 
If additionally  $\inj(M)>0$, then 
\be I_0(0^+):= \lim_{v\to0^+} I_0(v)=1,\ee
since in this case, all the limits at infinity (in the sense of Gromov-Hausdorff)  of the pointed manifolds of $(M,g)$ are $C^{0,\alpha}$ Riemannian manifolds, and one can use (for example) Theorem 1.3 in \cite{apps-mathann}.

We note that on an arbitrary Riemannian manifold, the function $I_M$  is at least  upper semicontinuous on $\big(0,\vol(M)\big)$ (see \cite{mfn}, Corollary 1), and hence Borel measurable; clearly, the same is true for $I_0$.

\bigskip

To go from isoperimetric profile expansion to small volumes inequalities (see Section~\ref{s5}) the first key tool is the following 
   \emph{generalized P\'olya-Szeg\H o inequality} on  an arbitrary smooth, connected, noncompact  Riemannian manifold:
   
   Let  $\Omega\subseteq M$ be open and with $\vol(\Omega)<\infty$.
   Then, for all $u\in C_c^\infty(\Omega)$ its Schwarz symmetrization $u^\#$ is Lipschitz on $\Omega^\#$, moreover for any $p\in[1,\infty)$  we have $u^\#\in W_0^{1,p}(\Omega^\#)$ and 
\be\int_{\Omega^\#} I_0(B_n |\xi|^n)^p |\nabla u^\#(\xi)|^p d\xi\le \int_{\Omega} |\nabla u(x)|^p d\mu(x),\label{ps}
\ee
where  $\Omega^\#$ is the ball in $\R^n$ centered at 0 with volume $|\Omega^\#|=\vol(\Omega)$. For a  discussion of this inequality see Proposition 1 and the Appendix in \cite{mnq}, where it is also used.

\bigskip

The second key tool is the following \emph{a priori  Talenti comparison theorem}: 

\begin{prop}\label{talenti-man} Let $\Omega\subseteq  M$ be open with $\vol(\Omega)<\vol(M)\le \infty$. 
 If  $u\in W_0^{1,2}(\Omega)$ and $f:=-\Delta u\in L^1(\Omega)$ in the sense of distributions, then  the following estimate holds:
 \be u^*(t)\le  n^{-2}B_n^{-\frac2 n}\int_t^{\vol(\Omega)} v^{-2+\frac 2n}I_0(v)^{-2} \int_0^v f^*(w)dwdv,\qquad 0<t\le \vol(\Omega). \label{estlap}\ee
 \end{prop}
 Here we adopt the usual convention that $1/0=\infty$ and $0\cdot\infty=0$, under which the right-hand side of \eqref{estlap} may be infinite. However, 
 as we mentioned earlier,    under $\ric\ge K$ and $ \inj(M)>0$ the profile   $I_M$ - and hence $I_0$ - is globally positive and H\"olderian,  so that the right-hand side of \eqref{estlap} is finite.
 
The above proposition  does not appear explicitly in the literature, but it follows as in the proof of Talenti's classical result in \cite{talenti1} on $\R^n$, which is essentially based on the Fleming-Rishel co-area formula, and the definition of the isoperimetric profile.  A sketch of proof is provided in the Appendix.

 We observe that the classical result in  \cite{talenti1}, Theorem 1, contains also  an existence and uniqueness statement for  the weak solution of $f=-\Delta u$ (in fact of $f=Lu$, where $L$ is a more general second order elliptic operator), assuming $f\in L^{\frac {2n}{n+2}}(\Omega)$ for $n\ge 3$ and $f\in L^p(\Omega)$ for $p>1$ if $n=2$, but the comparison estimate itself can be stated directly in terms of the function $u$, as in the above proposition.

\begin{rk} Note that on manifolds with $\ric\ge0$ and Euclidean volume growth we have $I_0(v)\searrow \sm^{\frac1n}>0$, as $v\to+\infty$ (see e.g. Theorem 3.8 in \cite{apps-mathann}), hence \eqref{ps} and \eqref{estlap} imply
\be\int_{\Omega^\#}  |\nabla u^\#(\xi)|^p d\xi\le\sm^{-{\frac pn}} \int_{\Omega} |\nabla u(x)|^p d\mu(x)\label{ps-evg}\ee

\be u^*(t)\le  \sm^{-\frac2n}n^{-2}B_n^{-\frac2 n}\int_t^{\vol(\Omega)} v^{-2+\frac 2n}\int_0^v f^*(w)dwdv=\sm^{-\frac2n}V^*(t)\label{estlap-evg}\ee
where $V(\xi)=V^\#(\xi)=V^*(B_n|\xi|^n)$ is the weak solution of $-\Delta V=f^{\#}$ in $\Omega^\#$, with $V=0$ on $\partial\Omega^\#.$
Estimate \eqref{ps-evg} appears in several papers, see for example \cite{bk}, \cite{apps-mathann}, \cite{nobili}, and references therein. Estimate \eqref{estlap-evg} was first derived  in \cite{cl}, see also \cite{mv} for RCD$(K,N)$ spaces with $K>0$, and \cite{wu1}, \cite{wu2} for the $p$-Laplacian on RCD spaces, including rigidity and almost-rigidity results. Proposition 1 differs in requiring no assumption on the isoperimetric profile or on the curvature, and in being a priori with $\Delta u\in L^1$.
\vskip1em

\vskip1em
\section{A local $I_1^\a$ inequality}\label{s3}

\bigskip
\begin{theorem}\label{druet} Assume $\ric\ge K$ and $\inj(M)>0$. Then, for any $\a\in(0,1)$ there exists $B>0$ and $r_0>0$ such that for any $x\in M$ and  all $u\in W_0^{1,1}(B(x,r_0))$
\be \|u\|_\nnn^\a\le K(n,1)^\a\big(\|\nabla u\|_1^\a+B\|u\|_1^\a\big).\label{sobdruet}
\ee
\end{theorem}

\begin{rk} The proof below will yield an explicit value of  $B$, depending on $n,\alpha,\rh^{}$ (see \eqref{B}).
\end{rk}
\smallskip

\smallskip
\ni{\bf Proof.} We follow closely the proof of Theorem 4 in \cite{mnq}, omitting or condensing details which are very similar (if not identical).    For $p\ge1$, $B>0$, $r>0$, and $x\in M$
\be \lambda_{p,r}(x)=\lambda_{p,r,g}(x)=\inf_{u\in W_0^{1,1}(B_g(x,r))\atop u\not\equiv 0}\frac{\|\nabla u\|_p^\a+B \|u\|_p^\a}{\|u\|_q^\a}\label{lambda}\ee

Fix any $B>0$, and assume that there is no $r_0>0$   so that the conclusion of Theorem~\ref{druet} is true. We will then derive an explicit upper bound on $B$ depending on $n$ and $\rh^{}$, see \eqref{B}.

Under such assumption, for all $r>0$  there is $y_r\in  M$ such that 
\be \lambda_{1,r}(y_r)<K(n,1)^{-\a}.\ee

Arguing as in \cite{mnq} , given $r_k\searrow0$ there is $p_k\searrow 1$ such that for $k$ large enough
\be \lambda_{p_k,r_k}(y_{r_k}^{})<K(n,1)^{-\a}.\ee

For the  rest of this proof  ``$p$'' will denote an element of the original  sequence $p_k\searrow 1,$ or a subsequence of it, and   the notation $p\to1$ will always mean ``up to a subsequence of $\{p_k\}$''. To be consistent with the notation used in \cite{mnq},  \cite{druet-pams}, and \cite{noa-arxiv}, given any subsequence of $\{p_k\}$ we will use $r_p$ to denote the corresponding subsequence  of $r_k\searrow 0$.

\smallskip
We will let 

\be\lambda_p=\ \lambda_{p,r_p}(y_p)\ee

so we  can then  assume that as $p\to 1$ 
\be \lambda_p< K(n,1)^{-\a} <K(n,p)^{-\a}.\label{L2}\ee
(for the last inequality see Remark 3 in \cite{mnq}). We recall here the explicit value of $K(n,p):$

\be K(n,p)=(p-1)^{1-\frac1p}n^{1-\frac1p}(n-p)^{-1+\frac1p}\bigg[\frac{\Gamma(n)}{\Gamma\big(\frac np\big)\Gamma\big(n+1-\frac np\big)}\bigg]^{\frac1n}\;K(n,1)\label{knp}\ee
where
\be K(n,1)=\frac{B_n^{-\frac1n}}n.\label{kn1a}\ee

\medskip
Let  $u_p$ be a minimizer of the functional in \eqref{lambda}, such that

\be u_p\in C^{1,\eta}\big(B(y_p,r_p)\big), \quad {\text {some} }\; \;\eta>0\ee

\be u_p>0 \;\;{\text {on} }\; \;B(y_p,r_p),\qquad u_p=0 \;\;{\text {on} }\;\;  \p B(y_p,r_p)\ee

\be -\Delta_p u_p+B\|u_p\|_p^{\a-p}\|\nabla u_p\|_p^{p-\a}u_p^{p-1}=\lambda_p\|\nabla u_p\|_p^{p-\a}u_p^{q-1}\label{E1}\ee

\be \|u_p\|_q=1,\label{u1}\ee

\be \|\nabla u_p\|_p^\a+B\|u_p\|_p^\a=\lambda_p.\label{eq2}\ee

where $\Delta_p={\rm {div} }\big(|\nabla u|^{p-2}\nabla u\big)$ denotes the $p$-Laplacian in the metric $g$.

If $x_p$ denotes the maximum point for $u_p$, and 

\be\mu_p^{1-n/p}:= u(x_p)=\max u_p,\ee

then the following hold true:

\be \lim_{p\to1}\mu_p=0,\label{p1}\ee

\be \lim_{p\to1} \|u_p\|_p=0,\label{p2}\ee

\be \lim_{p\to1}\lambda_p=K(n,1)^{-\a},\label{p3}\ee

\be \lim_{p\to1} \|\nabla u_p\|_p=K(n,1)^{-1}.\label{p4}\ee

For the proof see \cite{mnq} Proposition 1, the only difference is that for \eqref{p3}, \eqref{p4} we now use \eqref{L2}, \eqref{eq2}.

\medskip
Now we construct the blowup at the maximum points $x_p$,  in $C^{0,\a}$ harmonic coordinates. By results of Anderson-Cheeger \cite {ac} (see also the very nice paper by Hebey-Herzlich \cite{hebey-herzlich}, which was our main reference for this type of results), under $\ric\ge K$ and $\inj(M)>0$, the $C^{0,\a}-$harmonic radius $\rh^{}$ is positive, and we can find a harmonic chart
  $\Phi_p:B(x_p,\rh^{})\to \R^n$,  with $\Phi_p(x_p)=~0$, coordinates $\xi=\Phi_p(y)$, and such that if $g_{ij}^{(p)}(\xi)$ are the components of $g$ in that chart, 
  then
  \smallskip
\be g_{ij}^{(p)}(0)=\delta_{ij} \label{h1} \ee

\be \frac14\delta_{ij}\le g_{ij}^{(p)}\le 4 \delta_{ij} \quad {\text{as bilinear forms}} \label{h2} \ee

\be \rh^{\a} | g_{ij}^{(p)}|_{C^{0,\a}(B(x_p,\rh^{}))}\le 3.\label{h4} \ee
  
  In particular, note that if $y,z\in B(x_p,\rh^{})$, $\eta=\Phi_p(z)$, $\xi=\Phi_p(y)$, then
  \be\frac12 |\xi-\eta|\le d(y,z)\le 2|\xi-\eta|,\ee
  so if $y\in B(x_p,\rh^{})$ and $r<\rh^{}-d(x_p,y)$ then 
  \be B_{\sR^n}(\xi,\half r)\subseteq \Phi_p\big(B(y,r)\big)\subseteq B_{\sR^n}(\xi,2r)\label{balls0}\ee
  and in particular
   \be B_{\sR^n}(0,\half \rh^{})\subseteq \Phi_p\big(B(x_p,\rh^{})\big)\subseteq B_{\sR^n}(0,2\rh^{}).\ee
As a consequence, if $g_{(p)}^{ij}$ denote the component of the inverse matrix of $\{g_{ij}^{(p)}\}$ and if  
\be0<\dh<\min\Big\{\frac{\rh^{}} 2, \rh(3n^2)^{-\frac1{\a}}\Big\}= \rh(3n^2)^{-\frac1{\a}}\label{deltah}\ee
 then for $|\xi|<\dh$
\begin{align} |g_{ij}^{(p)}(\xi)-\delta_{ij}|&\le3\rh^{-\a}|\xi|^{\a}\label{gij1}\\
 |g_{(p)}^{ij}(\xi)-\delta_{ij}|&\le 6n\rh^{-\a}|\xi|^{\a},\label{gij2}\\
\big|\sqrt{|g_{(p)}|(\xi)}-1\big|&\le3 n^2 \rh^{-\a} |\xi|^{\a},\label{detg}\end{align}
where $|g_{(p)}|=\det\big(g_{ij}^{(p)}\big)$ (for details, see \cite{mnq}, Section 5).  Additionally, if we put the metric $g_{(p)}$ on $B_{\sR^n}(0,\frac12\rh^{})$, then \eqref{balls0} becomes
\be B_{\sR^n}(\xi,\half r)\subseteq B_{g_{(p)}}(\xi,r)\subseteq B_{\sR^n}(\xi,2r),\qquad |\xi|<\dh,\; r\le\dh\label{balls}\ee
where  $B_{g_{(p)}}$ denotes a ball in the metric $g_{(p)}$.

 For $p$ close enough to 1 we let
\be \Omega_p=\mu_p^{-1} \Phi_p(B(y_p,r_p)).\ee
Clearly (recall that $d(x_p,y_p)<r_p\to0$)
\be \Omega_p\subseteq \mu_p^{-1} \Phi_p(B(x_p,2r_p))\subseteq B_{\sR^n}(0,4\mu_p^{-1} r_p)\subseteq B_{\sR^n}(0,\mu_p^{-1}\dh)\subseteq  \R^n\label{omegap0}\ee
 Let 
(with some  abuse of language) $u_p(\xi)$ be $u_p(\Phi_p^{-1}(\xi))$, for $\xi\in \Phi_p(B(y_p,r_p))$, and 
\be v_p(\xi)=\bc\mu_p^{{\frac np}-1}u_p(\mu_p\xi ) & {\text{ if  }}\xi\in \Omega_p\\ 0 & {\text{ if }} \xi\in \R^n\setminus \Omega_p.\ec\ee

On $B_{\sR^n}(0,\mu_p^{-1}\dh)$, in particular on $\Omega_p$,  we consider the metric defined on $M$ as $\gp=\mu_p^{-2}g$, which in the coordinate chart given by $\mu_p^{-1}\Phi_p$,   is simply the tensor with components $g_{ij}^{(p)}(\mu_p\xi)$.  In the sequel $dV=dV_g$ denotes the volume element in the original metric $g$ on $M$, while the volume element in the metric $\gp$  on $\Omega_p$ will be denoted as
\be dV_\gp(\xi)=\sqrt{|g^{(p)}|(\mu_p \xi)}\;d\xi\label{dvp},\qquad |g^{(p)}|=\det\big(g_{ij}^{(p)}\big).\ee

\smallskip
\begin{rk}\label{amb}
In order to minimize notation a bit we will only emphasize the role of the metric $\gp$ when necessary. For 
example $\|\cdot\|_{p,\gp}$ will denote the $L^p$ norm of a function defined on $\Omega_p$  with respect to the measure in \eqref{dvp} while $\|u_p \|_p$  will continue to denote the norm of $u_p$ w.r. to the original metric $g$. Similarly, for a function $v$ defined on $\Omega_p$, the notation $|\nabla_\gp v|_{\gp}$  will mean that the gradient and the inner product are w.r. to $\gp$, while $|\nabla v|$ is just the Euclidean norm of the Euclidean gradient.
\end{rk}

\smallskip
Since $\Omega_p\subseteq B_{\sR^n}(0,\mu_p^{-1}\dh)$,   \eqref{gij1}, \eqref{gij2}, \eqref{detg}, \eqref{balls}, for $\xi\in \Omega_p$ can be rewritten as 

 \be  |(g_p)_{ij}(\xi)-\delta_{ij}|= |g_{ij}^{(p)}(\mu_p\xi)-\delta_{ij}|\le3\rh^{-\a}\mu_p^\a|\xi|^\a \le 12 \rh^{-\a} r_p^\a\label{gij12}\ee
 \be |g_p^{ij}(\xi)-\delta^{ij}|\le6n\rh^{-\a} \mu_p^\a|\xi|^\a\le 24n\rh^{-\a} r_p^\a\label{gij22}\ee
\be\Big|\sqrt{|g_p|(\xi)}-1\Big|\le 3n^2\rh^{-\a} \mu_p^\a|\xi|^\a\le 12 n^2 \rh^{-\a}r_p^\a, \label{detg2}\ee

where the first inequalities are true if we only have $\mu_p|\xi|<\dh$, and also
\be\frac12|\xi-\eta|\le d_{\gp}(\xi,\eta)\le 2|\xi-\eta|\ee
\be B_{\sR^n}(\xi,{\tfrac12}r)\subseteq B_{\gp}(\xi,r)\subseteq B_{\sR^n}(\xi,2r),\qquad |\xi|<\mu_p^{-1}\dh,\; r\le \mu_p^{-1}\dh\label{ballsp}\ee
(in particular for any $\xi,r$ if $p$ close enough to 1).

 \smallskip
If $\Delta_{p,\gp}$ denotes the $p-$Laplacian in the metric $\gp$, then we have the identities

\be \|v_p\|_{r,\gp}=\mu_p^{\frac n p-\frac n r-1}\|u_p\|_r\label{v1}\ee

\be\|\nabla_\gp v_p\|_{r,\gp}=\mu_p^{\frac n p-\frac n r}\|\nabla u_p\|_r\label{v2}\ee

\be- \Delta_{p,\gp} v_p+B\mu_p^\a\|v_p\|_{p,\gp}^{\a-p}\|\nabla_\gp v_p\|_{p,\gp}^{p-\a}v_p^{p-1}=\lambda_p\|\nabla_\gp v_p\|_{p,\gp}^{p-\a}v_p^{q-1}\label{E2}\ee
 In particular 
\be \|v_p\|_{q,\gp}=\|u_p\|_q=1\label{vpq}\ee
\be \|\nabla_\gp v_p\|_{p,\gp}=\|\nabla u_p\|_p.\label{pp}\ee

Now let
\be \vt_p=v_p^{\frac {p(n-1)}{n-p}}\label{vtilde}\ee
so that $\vt_p^\nnn=v_p^q$. We now claim that
\be \lim_{p\to1}\frac{\|\nabla \vt_p\|_1}{\|\vt_p\|_{\nnn}}=K(n,1)^{-1}.\label{sob1}\ee

Indeed from the sharp Sobolev inequality  on $\R^n$ we have $\|\nabla \vt_p\|_1\ge K(n,1)^{-1}{\|\vt_p\|_{\nnn}}$ which implies
\be \liminf_{p\to1}\frac {\|\nabla \vt_p\|_1}{\|\vt_p\|_{\nnn}}\ge K(n,1)^{-1}\ee

On the other hand, from 
 \eqref{h2},  \eqref{gij22}  and the definition of gradient, we have
\be \frac14|\nabla \vt_p|^2\le |\nabla_\gp\vt_p|_\gp^2\le 4|\nabla \vt_p|^2.\ee

\be\ba &\big||\nabla_\gp\vt_p|_\gp^2-|\nabla \vt_p|^2\big|\le6n\rh^{-\a} \mu_p^\a|\xi|^\a \bigg(\sum_{k=1}^n |\p_k \vt_p|\bigg)^2\cr&\le 6n^2\rh^{-\a}\mu_p^\a|\xi|^\a |\nabla \vt_p|^2\le  24n^2\rh^{-\a} \mu_p^\a|\xi|^\a |\nabla_\gp \vt_p|_\gp^2\ea\ee
hence
\be  |\nabla \vt_p|^2\le|\nabla_\gp\vt_p|_\gp^2\big(1+ 24n^2\rh^{-\a}\mu_p^\a|\xi|^\a \big)
\ee
and finally
\be |\nabla \vt_p|\le|\nabla_\gp\vt_p|_\gp\big(1+12 n^2\rh^{-1} \mu_p^\a|\xi|^\a \big).\ee
Using the above and \eqref{detg2}, \eqref{omegap0}, and the fact that, for $\xi\in \Omega_p$  (and $p$ close to 1)
\be{\frac12}\le \sqrt{|g_p|(\xi)}\le 2\label{volgp}\ee
we get
\be \Big| \sqrt{|g_p|(\xi)}-1\Big|
\le 3n^2\rh^{-\a}\mu_p^\a|\xi|^\a\le 6n^2\rh^{-\a}\mu_p^\a|\xi|^\a\sqrt{|g_p|(\xi)} \le 24n^2\rh^{-\a} r_p^\a\sqrt{|g_p|(\xi)}
\label{vest}\ee
and
\be\ba \|\nabla \vt_p\|_1&\le\int_{\R^n} |\nabla_\gp \vt_p|_\gp\big(1+ 12n^2 \rh^{-\a}\mu_p^\a|\xi|^\a \big)\big(1+3n^2 \rh^{-\a}\mu_p^\a|\xi|^\a\big)dV_\gp(\xi)\cr&
\le \int_{\R^n} |\nabla_\gp \vt_p|_\gp\big(1+16n^2\rh^{-\a}\mu_p^\a|\xi|^\a\big)dV_\gp(\xi)
\le \|\nabla_\gp \vt_p\|_{1,\gp}\big(1+64n^2\rh^{-\a} r_p^\a\big)\ea\label{1}\ee
for $p$ close enough to 1. Arguing as in the proof of \eqref{p3}
\be\ba \|\nabla_\gp \vt_p\|_{1,\gp}&=\frac{p(n-1)}{n-p}\int_{\R^n}|\nabla_\gp v_p|_\gp v_p^{\frac {n(p-1)}{n-p}}dV_\gp(\xi)\cr&\le \frac{p(n-1)}{n-p}\|\nabla_\gp v_p\|_{p,\gp}\|v_p\|_{q,\gp}=
 \frac{p(n-1)}{n-p}\|\nabla u_p\|_p\cr&= \frac{p(n-1)}{n-p}(\lambda_p-B\|u_p\|_p^\a)^{1/\a}\to K(n,1)^{-1}
\ea\label{2}
\ee
which together with 
\be\ba\|\vt_p\|_{\nnn}&=\int_{\R^n}v_p^q\;d\xi\ge\int_{\R^n}v_p^q\big(1-6n^2\rh^{-\a}\mu_p^\a|\xi|^\a)\big)dV_\gp(\xi)\cr&\ge (1-24n^2\rh^{-\a} r_p^\a)\int_{\R^n}v_p^qdV_\gp(\xi)\to1 \ea\label{3} \ee
(where we used \eqref{vest}) yields $\limsup_{p\to1}{\|\nabla \vt_p\|_1}/{\|\vt_p\|_{\nnn}}\le K(n,1)^{-1}$, and hence \eqref{sob1}.

Now, \eqref{sob1} implies that $\{\vt_p\}$ is uniformly bounded in $W^{1,1}(\R^n)$ and therefore there is $v_0\in BV(\R^n)$ such that $\vt\to v_0$ weakly in $BV(\R^n)$, i.e. $\vt_p\to v_0$ in $L^1(K)$ for any $K$ compact, and $\int \phi\cdot \nabla \vt_p d\xi\to\int \phi\cdot Dv_0$, for any $\phi$ smooth, compactly supported and valued in $\R^n$.

Proceeding as in \cite{druet-pams} and \cite{noa-arxiv},  we get that $\vt_p\to \chi_{B_{\R^n}^{}(\xi_0,R_0)}$ strongly in $BV(\R^n)$, (in particular in $L^\nnn(\R^n)$) where $\xi_0\in\R^n$ and 
\be |B_{\R^n}^{}(0,R_0)|=B_nR_0^n=1.\ee

As it turns out,

 \be |\xi_0|\le 4R_0,\label{xi0}\ee

and $v_p$ satisfies the following weak estimate:  for  $p$ close enough to 1 we have
\be v_p(\xi)\le \min\bigg\{1, \bigg(\frac{8R_0}{|\xi|}\bigg)^{\frac{n}{p}-1}\bigg\},\qquad \xi\in\R^n.\label{decay} \ee

The  proofs of  these estimates are contained in the  proof of Lemma 2 in \cite{mnq}, where they were proved in $C^1$ harmonic charts, but they go through unscathed in our $C^{0,\alpha}$ setting.

\medskip
 From the Sobolev inequality we have
\be\frac{\|\vt_p\|_\nnn}{\|\nabla\vt_p\|_1}\le K(n,1).\label{d8}\ee

\medskip
From \eqref{vpq}, \eqref{vtilde}, and  \eqref{detg2}  we have 
\be\ba\|\vt_p\|_{\nnn}&\ge\int_{\R^n}v_p^q\Big(\sqrt{|g_p|(\xi)}-3n^2\rh^{-\a}\mu_p^\a|\xi|^\a\Big)d\xi =1-3n^2\rh^{-\a} \mu_p^\a\int_{\R^n} |\xi|^\a v_p^q d\xi\ea\ee
and
\be\ba\int_{\R^n} |\xi|^\a v_p^q d\xi=\int_{\R^n} |\xi|^\a v_p^{q/p} v^{q(1-1/p)} d\xi
&\le\bigg(\int_{\R^n} |\xi|^{\a p} v_p^q d\xi\bigg)^{\frac1p}\bigg(\int_{\R^n}  v_p^q d\xi\bigg)^{\frac1{p'}}\cr&
\le 2^{\frac 1{p'}}\bigg(\int_{\R^n} |\xi|^{\a p} v_p^q d\xi\bigg)^{\frac1p}. \ea
\ee
so
\be\ba\|\vt_p\|_{\nnn}&\ge1-3\cdot 2^{\frac1{p'}}n^2\rh^{-\a} \mu_p^\a\bigg(\int_{\R^n} |\xi|^{\a p} v_p^q d\xi\bigg)^{\frac1p}. \label{vnorm}\ea\ee
 From \eqref{1}, \eqref{2}, \eqref{L2} we have
\be\ba \|\nabla \vt_p\|_1
&\le \int_{\R^n} |\nabla_\gp \vt_p|_\gp\big(1+16n^2 \rh^{-\a}\mu_p^\a|\xi|^\a \big)dV_\gp\cr&= \frac{p(n-1)}{n-p}\big(\lambda_p-{B}\|u_p\|_p^{\a}\big)^{\frac1\a}+16n^2\rh^{-\a}\mu_p^\a \int_{\R^n}  |\xi|^\a |\nabla_\gp \vt_p|_\gp dV_\gp\cr&=
\frac{p(n-1)}{n-p}\big(\lambda_p-{B}\mu_p^\a\|v_p\|_{p,\gp}^\a\big)^{\frac1\a}+16n^2\rh^{-\a}\mu_p^\a \int_{\R^n}  |\xi|^\a |\nabla_\gp \vt_p|_\gp dV_\gp
\cr&<K(n,1)^{-1}-K(n,1)^{-1+\a}B\mu_p^\a\|v_p\|_{p,\gp}^\a+16n^2\rh^{-\a}\mu_p^\a \int_{\R^n}  |\xi|^\a |\nabla_\gp \vt_p|_\gp dV_\gp.
\ea\label{nablavnorm1}
\ee

 Note that 

\be\ba &\int_{\R^n}  |\xi|^\a |\nabla_\gp \vt_p|_\gp dV_\gp=\frac{p(n-1)}{n-p}\int_{\R^n}  |\xi|^\a v_p^{\frac{n(p-1)}{n-p}}  |\nabla v_p|_\gp dV_\gp\cr&\le 2(\int_{\R^n} v_p^{p'\frac{n(p-1)}{n-p}} dV_{\gp}\bigg)^{\frac1{p'}}\bigg(\int_{\R^n}  |\xi|^{\a p} |\nabla_{\gp} v_p|_{\gp}^p dV_\gp\bigg)^{\frac1p} \cr&=2\bigg(\int_{\R^n}  |\xi|^{\a p} |\nabla_{\gp} v_p|_{\gp}^p dV_\gp\bigg)^{\frac1p}.\label{d0}
\ea\ee
so \eqref{nablavnorm1}  implies
\be\|\nabla \vt_p\|_1
\le K(n,1)^{-1}-K(n,1)^{-1+\a}B\mu_p^\a\|v_p\|_{p,\gp}^\a+32n^2\rh^{-\a}\mu_p^\a \bigg(\int_{\R^n}  |\xi|^{\a p}  |\nabla_\gp v_p|_\gp^p dV_\gp\bigg)^{\frac1 p}.
\label{nablavnorm2}
\ee

In the introduction we stated that the main technical difficulty we had to bypass, going from $C^1$ charts to $C^{0,\alpha}$ charts,  was the use of the strong pointwise estimate \eqref{strong}.  The next lemma is precisely what we need:

\begin{lemma}\label{lemma} For any $\delta\in [0,1]$  we have 
 
\be \int_{\R^n}|\xi|^{\delta p}v_p^qd\xi\le2^{1-\d} (8R_0)^{\delta p}\bigg(\int_{\R^n} v_p^p d\xi\bigg)^\delta\le 2(8R_0)^{\delta p} \|v_p\|_{p,\gp}^{\delta p}\label{dp1}\ee
and
\be \int_{\R^n}  |\xi|^{\a p}  |\nabla_\gp v_p|_\gp^p dV_\gp\le 34n^{2\a-1}R_0^{1-\a}\|v_p\|_{p,\gp}^{\alpha p}.\label{dp2}\ee

\end{lemma}

\begin{rk}\label{weak-strong} We note that in \cite{mnq}, working  with positive $C^1$ harmonic radius, we used the strong estimate
\be v_p(\xi)\le \min\bigg\{1, \bigg(\frac{8R_0}{|\xi|}\bigg)^{\frac{n-p-b}{p-1}}\bigg\},\qquad \xi\in\R^n.\label{decay-1} \ee
and obtained that for $\delta=1/p$  the integral on the left hand side of \eqref{dp1} is bounded by $2\cdot 8^{n+1}R_0$,  and the integral on the left hand side of \eqref{dp2} (when $\alpha=1$) is bounded by $16n 8^n$. We also used that $\|v_p\|_{p,\gp}=1+o(1)$, as $p\to 1$, which follows from the strong estimate \eqref{decay-1} (but not from the weak estimate \eqref{decay}).

\bigskip
\ni{\bf Proof.}  Below we will use repeatedly the volume comparison \eqref{volgp}, and
\be \|\nabla_{\gp}v_p\|_{p,\gp}=\big(\lambda_p-B\mu_p^\a\|v_p\|_{p,\gp}^\a\big)^{\frac1\a}.\ee
 
It's enough to prove \eqref{dp1} for $\d\in (0,1).$  
We have

\be\ba\int_{\R^n } |\xi|^{\d p} v_p^q d\xi&=\int_{\R^n } |\xi|^{\d p} v_p^{q-\d p}v_p^{\delta p} d\xi\cr&
\le \bigg(\int_{\R^n } |\xi|^{\frac{\d p}{1-\d}  } v_p^{\frac {q-\d p}{1-\d}}d\xi \bigg)^{1-\d} 
\bigg(\int_{\R^n } v_p^pd\xi \bigg)^{\d}\cr&=  \bigg(\int_{\R^n } |\xi|^{\frac{\d p}{1-\d}  } v_p^{\frac {q-\d p}{1-\d}-q}v_p^qd\xi \bigg)^{1-\d} \bigg(\int_{\R^n } v_p^pd\xi \bigg)^{\d}\cr
\ea\ee 
using \eqref{decay} we have
\be   |\xi|^{\frac{\d p}{1-\d}  } v_p^{\frac {q-\d p}{1-\d}-q}\le |\xi|^{\frac{\d p}{1-\d}  }\bigg (\frac{8R_0}{|\xi|}\bigg)^{\frac {\d(q-p)}{1-\d}{(\frac n p-1)}}=(8R_0)^{\frac {\delta p}{1-\d}}\ee
so that 
\be\int_{\R^n } |\xi|^{\d p} v_p^q d\xi\le (8R_0)^{\delta p} \bigg(\int_{\R^n } v_p^qd\xi \bigg)^{1-\d}\ \bigg(\int_{\R^n } v_p^pd\xi \bigg)^{\d}\ee
which gives \eqref{dp1}.

To prove \eqref{dp2}, let's start by  multiplying \eqref{E2} by $|\xi|^{p} v_p$, integrating by parts,  and using \eqref{dp1} with $\delta=1$, obtaining

\be\ba &\int_{\R^n} |\nabla_\gp v_p|_\gp^{p-2}\langle\nabla_\gp (|\xi|^{p} v_p),\nabla_\gp v_p\rangle_\gp dV_\gp\cr
&=\lambda_p\|\nabla_\gp v_p\|_{p,\gp}^{p-\a}\int_{\R^n} |\xi|^{ p}v_p^q dV_\gp-B\mu_p^\a\|v_p\|_{p,\gp}^{\a-p}\|\nabla_\gp v_p\|_{p,\gp}^{p-\a}\int_{\R^n} |\xi|^{ p} v_p^pdV_\gp
\cr&\le 
\lambda_p\|\nabla_\gp v_p\|_{p,\gp}^{p-\a}\int_{\R^n} |\xi|^{ p}v_p^q dV_\gp\le
 2\lambda_p^{\frac p\a}\int_{\R^n} |\xi|^{ p}v_p^q d\xi\cr&\le 4(8R_0)^{ p}\lambda_p^{\frac p\a}\|v_p\|^{ p}_{p,g_p}\le 33 n  \|v_p\|^{p}_{p,g_p},\label{vpxi3}
\ea\ee
where we used 
$$\lambda_p\to K(n,1)^{-\a}=\Big(\frac n {R_0}\Big)^{\a}.$$

On the other hand (where as usual $\xi^*=\xi/|\xi|$)
\be\ba&\int_{\R^n} |\nabla_\gp v_p|_\gp^{p-2}\langle\nabla_\gp (|\xi|^{ p} v_p),\nabla_\gp v_p\rangle_\gp dV_\gp=\int_{\R^n}
| \nabla_\gp v_p|_\gp^p |\xi|^{ p}dV_\gp\cr&+ p\int_{\R^n} |\xi|^{ p-1} v_p  |\nabla_\gp v_p|_\gp^{p-2}\langle \xi^*, \nabla_\gp v_p\rangle_\gp dV_\gp \ge \int_{\R^n}
| \nabla_\gp v_p|_\gp^p |\xi|^{p}dV_\gp\cr&-p\int_{\R^n} v_p |\xi|^{ p-1}   |\nabla_\gp v_p|_\gp^{p-1} dV_\gp.\label{d4}
 \ea\ee

From H\"older's and Young's inequalities  
\be \ba p\int_{\R^n} v_p|\xi|^{p-1}  |\nabla_\gp v_p|_\gp^{p-1} dV_\gp&\le p \|v_p\|_{p,\gp}\bigg(\int_{\R^n} |\xi|^{p}  |\nabla_\gp v_p|_\gp^{p} dV_\gp\bigg)^{\frac {p-1}p}\cr&
\le \frac{p^p \|v_p\|_{p,\gp}^p} p+\frac{1}{p'}\int_{\R^n} |\xi|^{p}  |\nabla_\gp v_p|_{p,\gp}^{p} dV_\gp
\label{d5}\ea\ee
then, combining \eqref{vpxi3}, \eqref{d4}, \eqref{d5} we get
\be \int_{\R^n} |\xi|^{p}  |\nabla_\gp v_p|_\gp^{p} dV_\gp\le (p^p+ 33pn) \|v_p\|_{p,\gp}^p.
\label{d5a}
\ee

One more application of H\"older's inequality gives

\be\ba &\int_{\R^n} |\xi|^{\a p}  |\nabla_\gp v_p|_\gp^{p} dV_\gp=
 \int_{\R^n} |\xi|^{\a p}  |\nabla_\gp v_p|_\gp^{\a p}  |\nabla_\gp v_p|_\gp^{p(1-\a)} dV_\gp\cr&
 \le \bigg( \int_{\R^n} |\xi|^{p}  |\nabla_\gp v_p|_\gp^{p} dV_\gp\bigg)^{\a} \bigg( \int_{\R^n}   |\nabla_\gp v_p|_\gp^{p} dV_\gp\bigg)^{1-\a}\cr&
 \le \big(p^p+ 33pn\big)^\a\|v_p\|_{p,\gp}^{\a p}  \|\nabla_\gp v_p\|_\gp^{p(1-\a)}.
  \ea
\ee
\be\ba (p^p+ 33pn)^\a\|\nabla_\gp v_p\|_\gp^{p(1-\a)}&\to (1+ 33n)^\a(R_0/n)^{1-\a}<34n^{2\a-1}R_0^{1-\a},
 \ea\ee
 which proves \eqref{dp2}.
 
 \qed

From  \eqref{vnorm} and \eqref{dp1} (with $\d=\alpha$) we get
\be\ba \|\vt_p\|_{\nnn}&\ge 1-3\cdot 2^{\frac 1{p'}}n^2\rh^{-\a} \mu_p^\a 2^{\frac 1p}(8R_0)^\a \|v_p\|_{p,\gp}^{\a }\cr&\ge 1-50n^2\rh^{-\a} R_0^\a\mu_p^\a  \|v_p\|_{p,\gp}^{\a }
:=1-H_1\mu_p^\a \|v_p\|_{p,\gp}^{\a}\label{vnorm2}
\ea\ee

and from \eqref{nablavnorm2} and \eqref{dp2} 
\be\ba\|\nabla \vt_p\|_1
&\le K(n,1)^{-1}-K(n,1)^{-1+\a}B\mu_p^\a\|v_p\|_{p,\gp}^\a+32n^2\rh^{-\a}\mu_p^\a (34n^{2\a-1}R_0^{1-\a}\|v_p\|_{p,\gp}^{\alpha p})^{\frac1 p}\cr&
<K(n,1)^{-1}-K(n,1)^{-1+\a}B\mu_p^\a\|v_p\|_{p,\gp}^\a+34^2 n^{2\a+1}R_0^{1-\a}\rh^{-\a}\mu_p^\a\|v_p\|_{p,\gp}^{\alpha }\cr&=K(n,1)^{-1}\Big(1+\mu_p^\a \|v_p\|_{p,\gp}^{\alpha }\big(-B K(n,1)^{\a}+34^2 n^{2\a+1}R_0^{1-\a}K(n,1)\rh^{-\a}\big)\Big)
\cr&:=K(n,1)^{-1}\big(1+H_2 \mu_p^\a \|v_p\|_{p,\gp}^{\alpha }\big)
\label{nablavnorm3}
\ea\ee

Recall  that by \eqref{p2} and \eqref{v1}
\be \mu_p\|v_p\|_{p,g_p}\to 0,\ \ \quad \text{as}\ p\to 1^+.\ee

Combining \eqref{vnorm2} and \eqref{nablavnorm3} we get
\be\ba K(n,1)&\ge\frac{\|\vt_p\|_\nnn}{\|\nabla\vt_p\|_1}> K(n,1)\frac{ 1-H_1\mu_p^\a\|v_p\|_{p,\gp}^\a}{1+H_2\mu_p^\a\|v_p\|_{p,\gp}^\a}\cr&
=K(n,1)\bigg(1-\frac{(H_1+H_2)\mu_p^\a\|v_p\|_{p,\gp}^\a}{1+H_2 \mu_p^\a\|v_p\|_{p,\gp}^\a}\bigg)\ea
\ee
which implies $H_1+H_2>0$ i.e.
\be\ba B&< 34^2 n^{2\a+1}R_0^{1-\a}K(n,1)^{1-\a}\rh^{-\a}+50 n^2 R_0^\a K(n,1)^{-\a}\rh^{-\a}\cr&
=\big( 34^2 n^{3\a} R_0^{2-2\a}+50n^{2+\a}\big)\rh^{-\a}\label{B}
\ea\ee
Hence, if we take  $B$ equal to  the right-hand side of \eqref{B}, then we can find $r_0>0$ so that \eqref{sobdruet} holds for all $u\in W_0^{1,1}(B(x,r_0))$ and all $x\in M$.

\qed
\medskip
\begin{rk} We would like to highlight the fact that one cannot obtain the local $I_1^1$ inequality with the above proof, assuming only $\ric\ge K$ and $\inj(M)>0$. Indeed, assuming  that for a given $B$ the local $I_1^1$ fails, would yield a minimizing sequence $\{u_p\}$ satisfying 
\be\|u_p\|_p+B\|u_p\|_p=\lambda_p.\label{xx}\ee
However, once a $C^{0,\a}$ chart is chosen, following the above proof (especially the steps inside \eqref{nablavnorm1}), equation \eqref{xx} would yield 
\be\|\nabla \vt_p\|_1<K(n,1)^{-1}-B\mu_p\|v_p\|_{p,\gp}+C\mu_p^\a\|v_p\|_{p,\gp}^\a\ee
in place of \eqref{nablavnorm3}, which clearly does not allow the argument to work.

\vskip1em


\section{From local $I_1^\a$ to profile expansion}\label{s4}

\vskip1em

In this section we show how the local $I_1^\a$ inequality implies a first order expansion for the isoperimetric profile.

\begin{theorem}\label{expansion} Suppose that on $M$ we have $\ric\ge K$ and $\inj(M)>0$. Then, for any $\a\in(0,1)$ there exist $\tau_0,B>0$ such that  
\be \im(v)\ge nB_n^{\frac1n} v^\nn- B v^{\frac{n-1+\a}n}\qquad 0<v\le \tau_0.\label{cond1}\ee
\end{theorem}

\bigskip
\ni{\bf Proof.} The proof is essentially the same as the one of  Theorem 4 in \cite{mnq}, so it will be only outlined here. Basically, all we need to do from that proof is to replace $C^{1,\alpha}$ with $C^{0,\alpha}.$ 

First, from Theorem \ref{druet} we obtain (by smooth approximation of $\chi_\Omega$)  that under the given hypothesis
there exist $B>0$ and $r_0>0$ such that for any set $\Omega\subseteq M$ of finite perimeter and with $\diam (\Omega)\le r_0$ we have
\be \per(\Omega)^\a\ge K(n,1)^{-\a} \vol(\Omega)^{\a\frac{n-1}n}-B\;\vol(\Omega)^{\a },
\label{isop0}\ee
and this easily implies 
\be \ba
 \per(\Omega)&\ge K(n,1)^{-1} \vol(\Omega)^{\frac{n-1}n}-\frac B\a K(n,1)^{\a-1}\;\vol(\Omega)^{1-\frac1n+\frac\a n}\cr&=nB_n^{\frac1n}\vol(\Omega)^{\frac{n-1}n}-B'\;\vol(\Omega)^{\frac{n-1+\a}n}.
\ea\label{isop1}\ee

 Under the given assumptions on the geometry,  there exists $v_0>0$ such that for $0<v<v_0$ at least one of the following two cases  holds:
 
\medskip
\ni(A) There exists $\Omega\subseteq M$ with finite perimeter such that 
\be \vol(\Omega)=v,\qquad \im(v)=\per(\Omega)\label{isregion}\ee

\smallskip
\ni(B) There exists a pointed $C^{0,\alpha}$ Riemannian manifold $(M_0, x_0, g_0)$ and a sequence $\{x_j\}$ on $M$ such that $(M,x_j,g)\to (M_0,x_0,g_0)$ in the $C^{0,\a}$ topology. Moreover, there exists  $\Omega_0\subseteq M_0$ with finite perimeter such that 
\be \vol_{g_0}(\Omega_0)=v,\qquad \im(v)=\per_{g_0}(\Omega_0)\label{isregion0}\ee
(see  \cite{apps-annales} Lemma 4.18).

Moreover, there exist $C_0>0$, and $v_0>0$ such that any isoperimetric region $\Omega$ satisfies  
\be \diam(\Omega)\le C_0 \vol(\Omega)^{\frac1n},\qquad {\text{ whenever }}\;\vol(\Omega)\le v_0\label{diam-vol}\ee
(see \cite{apps-annales} Prop. 4.21), valid in the general context of RCD spaces, which includes Riemannian manifolds satisfying $\ric\ge K$ and $\inj(M)>0$, and their pointed limits.

With this in mind, it is straightforward to find $\tau_0$ such that for $\vol(\Omega)=v\le \tau_0$: 

 1) if  (A) occurs, then  \eqref{cond1} follows  immediately from \eqref{isop1};
 
  2)  if (B) occurs, then  \eqref{cond1} follows from \eqref{isop1} applied to $\Omega_0$, which in turn follows from \eqref{isop1}  applied to the domains $\Omega_j$ in $M$ corresponding to $\Omega_0$ in the pointed $C^{0,\a}$ convergence.

To clarify this last point, recall that $(M,x_j,g)\to (M_0,x_0,g_0)$ in the $C^{0,\a}$ topology means the following: for any compact set $K\subseteq M_0$, with $x_0\in K$ there exist, up to a subsequence, compact sets $K_j\subseteq M$, $x_j\in K_j$, and $C^{1,\alpha}$ diffeomorphisms $\Phi_j:K\to K_j$ such that $\Phi_j^{-1}(x_j)\to x_0$ and $\Phi_j^* g$   converges in $C^{0,\alpha}$  to $g_0$, in the induced $C^{1,\alpha}$ complete atlas of $K$.
Take  $K=\overline {B(x_0,R)}$, where $R$ is large enough so as to contain $\Omega_0$ (which is bounded by \eqref{diam-vol}),  and let $\Omega_j=\Phi_j(\Omega_0)$. Then,
\be \diam(\Omega_j)  \to \diam(\Omega_0),\qquad \vol(\Omega_j)\to \vol_{g_0}(\Omega_0),\qquad \per(\Omega_j)\to \per_{g_0}(\Omega_0),\label{seq}\ee
which follow as in \cite{nardulli-calcvar} proof of  Lemma 3.3.  Choosing $\tau_0$ small enough, we then have $\diam(\Omega_j)$ small enough so that 
\be \per(\Omega_j)\ge nB_n^{\frac1n} \vol(\Omega_j)^{\frac{n-1}n}-B'\;\vol(\Omega_j)^{\frac{n-1+\a} n},
\ee
and passing to the limit as $j\to\infty$ we obtain \eqref{isop1} applied to $\Omega_0$.

\qed

\vskip1em
\section{From  profile expansion to small volumes inequalities}\label{s5}

\vskip1em

In this section we show how a first order expansion of the isoperimetric profile for small volumes, implies the sharp Sobolev inequality for small volumes.

 We will say that given $A>0$ the  \emph{Sobolev inequality with constant $A$ holds for small volumes on $M$}, if, given $p,q$ in \eqref{par} and $\a>0$, there are $\tau>0$ and  $B\ge 0$ such that for for any open set $\Omega$ with $\vol(\Omega)\le \tau$, and for all $u\in W_0^{1,p}(\Omega)$, we have
\be \|u\|_q^\a\le A\|\nabla u\|_p^\a+B\|u\|_p^\a.\label{sob}\ee
We will say that the \emph{Sobolev inequality with constant $A$ holds globally on $M,$} if there is $B\ge0$ such that \eqref{sob} holds for all $u\in W^{1,p}(M).$

Likewise, for given  $k=1,2$, $k<n$,  and $\gamma>0$ we will say that the   \emph{Moser-Trudinger inequality with constant $\gam$ holds for small volumes on $M$}, if for any $\kappa>0$ there is $C>0$ such that for any open set $\Omega$ with $\vol(\Omega)\le \tau$, and 
\be\bc {\text {for all}} \; u\in W_0^{1,n}(\Omega) & {\text {if}} \;k=1 \cr\cr
{\text{for all}} \; u\in  W_0^{1,2}(\Omega)\cap W^{2,\frac n 2}(\Omega) & {\text{if}} \; k=2 \ec\ee
 with 
\be \|\nabla^k u\|_{n/k}\le1,\label{mtl}\ee
we have 
\be  \int_{\Omega} \exp_{\llceil\frac{n-k}{k}-1\rrceil}{\left(\gam|u|^{\frac{n}{n-k}}\right)}d\mu\le C. \label{MTL}\ee

 We will say that the   \emph{Moser-Trudinger inequality with constant $\gam$ holds globally on $M$}, 
if there is $C>0$ such that for all $u\in W^{k,\frac nk}(M)$, estimate  \eqref{MTL} holds with $\Omega=M$, under \eqref{ruf}.

\begin{rk}\label{sobspace} We observe that for all $n\ge2$  we have $C_c^\infty(\Omega)\subseteq  W_0^{1,2}(\Omega)\cap W^{2,\frac n2}(\Omega)$, and for $n\ge4$ we have $W_0^{2,\frac n 2}(\Omega)\subseteq W_0^{1,\frac n2}(\Omega)\cap W^{2,\frac n 2}(\Omega)\subseteq W_0^{1, 2}(\Omega)\cap W^{2,\frac n 2}(\Omega)$. When $n=3$ 
 in general $W_0^{2,\frac 3 2}(\Omega)\not\subseteq W_0^{1, 2}(\Omega)\cap W^{2,\frac 3 2}(\Omega)$,
unless $\Omega$ has smooth boundary (in which case $ W_0^{1,\frac 32}(\Omega)\cap W^{2,\frac 3 2}(\Omega)= W_0^{1,2}(\Omega)\cap W^{2,\frac 32}(\Omega)$).

\end{rk}

\begin{theorem}\label{smallprofile-sobolev}
Assume that on a  complete Riemannian manifold $I_0(0^+)=1$, and  there exist constants $\tau_0>0$, $B\ge0$, $0<b<1$ such that 
\be I_0(v)\ge 1- B v^{\frac bn},\qquad 0<v\le \tau_0.\label{cond}\ee
Then for any $a\in(0,b)$ the Sobolev inequality \eqref{sob} with constant $K(n,p)^a$ holds for small volumes on $M$, in particular, there is $\tau\in (0,\tau_0]$ such that for any $\Omega$ open with $\vol(\Omega)\le \tau$
\be \|u\|_q^a\le K(n,p)^a\|\nabla u\|_p^a+B_1\|u\|_p^a,\qquad u\in W_0^{1,p}(\Omega)\label{sob3} \ee
where 
\be B_1=8B\tau^{\frac{b-a}{n}}+8BB_n^{\frac 1n}\Big(\frac{1-a}{b-a}\Big)^{\frac{1-a}n}\tau^{\frac{b-a}n}K\Big(n,\frac{qn}{an+q}\Big).\ee
\end{theorem}

\begin{rk}\label{I0} The condition $I_0(0^+)=1$ is only used to bound $I_0(v)$ from above for small~$v$. We stated it this way since it holds under reasonable hypothesis such as $\ric\ge K$, and noncollapsing volumes (see Section \ref{s2}).

\smallskip
\ni {\bf Proof.}
 Let us pick $\tau\le \tau_0$ so small so that 
\be 1-B\tau^{\frac bn}\ge  \frac12,\qquad   I_0(v)\le 2,\qquad {\text {for}} \;\;0\le v\le \tau.\label{tt}\ee
It is enough to show \eqref{sob3} for  $u\in C_c^\infty(\Omega)$, with $\Omega$ open and $\vol(\Omega)\le \tau.$ 

Since $u\in C_c^\infty(\Omega)$, and $\mu(\Omega)\le\tau<\infty$,  we have that  $u^\#$ is Lipschitz, and supported   on $\Omega^\#\subseteq B_{\sR^n}(0,R)$, where $B_n R^n=\tau$ (see  \cite{mnq}, Proposition 1). Hence,
\be\ba &u^\#(\xi)=\int_{|\xi|}^R (-\partial_\rho) u^\#(\rho\xi^*)d\rho=\int_{|\xi|}^RI_0(B_n\rho^n)^{-1}I_0(B_n\rho^n) (-\partial_\rho) u^\#(\rho\xi^*)d\rho\cr&\le \int_{|\xi|}^R(1+2BB_n^{\frac bn}\rho^b) I_0(B_n\rho^n) (-\partial_\rho) u^\#(\rho\xi^*)d\rho
\le  v(\xi)+4BB_n^{\frac bn}\int_{|\xi|}^R \rho^b (-\partial_\rho) u^\#(\rho\xi^*)d\rho\label{repformula1}
\ea
\ee
where we set $\xi=|\xi|\xi^*$ and 

\be v(\xi)=\int_{|\xi|}^R I_0(B_n\rho^n)(-\partial_\rho)u^\#(\rho\xi^*)d\rho.\ee

Note that 
\be\ba u^\#(\xi)^a&\le \bigg(v(\xi)+4BB_n^{\frac bn}\int_{|\xi|}^R \rho^b (-\partial_\rho) u^\#(\rho\xi^*)d\rho\bigg)^a\cr
&\le v(\xi)^a+4aBB_n^{\frac bn}v(\xi)^{a-1}\int_{|\xi|}^R \rho^b (-\partial_\rho) u^\#(\rho\xi^*)d\rho\cr
&=v(\xi)^a+4aBB_n^{\frac bn}v(\xi)^{a-1}|\xi|^b u^\#(\xi)+4aBB_n^{\frac bn}v(\xi)^{a-1}\int_{|\xi|}^R \rho^{b-1} u^\#(\rho\xi^*)d\rho
\ea
\ee
It is clear that
\be v(\xi)\geq \frac 12 u^\#(\xi)\ge  \frac 12 u^\#(\rho \xi^*),\qquad |\xi|\le\rho\le R \ee
so we get
\be\ba u^\#(\xi)^a&\le v(\xi)^a+8BB_n^{\frac bn}|\xi|^b u^\#(\xi)^a+8BB_n^{\frac bn}\int_{|\xi|}^R  \rho^{b-1}u^\#(\rho\xi^*)^a d\rho\cr
\label{ua1}
\ea\ee
From the \PS inequality \eqref{ps} we  have
\be\ba \int_{B(0,R)}|\nabla v(\xi)|^p d\xi&=\int_{B(0,R)}I_0(B_n|\xi|^n)^p|\nabla u^\#(\xi)|^p d\xi \le\int_{\Omega}|\nabla u(x)|^p d\mu(x).\label{psv}\ea\ee
From Minkowski's inequality and \eqref{ua1}
\be\ba\|u\|_q^a&=\bigg(\int_{B(0,R)} \big(u^\#(\xi)^a\big)^{\frac qa}d\xi \bigg)^{\frac aq}\le\bigg(\int_{B(0,R)}v(\xi)^qd\xi \bigg)^{\frac aq}+8BB_n^{\frac bn}\bigg(\int_{B(0,R)}|\xi|^{\frac{bq}{a}} u^\#(\xi)^qd\xi \bigg)^{\frac aq}\cr&+8BB_n^{\frac bn}\Bigg(\int_{B(0,R)} \bigg(\int_{|\xi|}^R \rho^{b-1}u^\#(\rho\xi^*)^ad\rho\bigg)^{\frac qa} d\xi \Bigg)^{\frac aq}:=I+II+III\ea\ee
From \eqref{psv} and the sharp Sobolev inequality on $\R^n$ we have
\be I=\|v\|_q^a\le K(n,p)^a\|\nabla v\|_p^a\le K(n,p)^a\|\nabla u\|_p^a.\label{I}\ee

To estimate the second term, first note that $u^\#(\xi)\le \|u\|_p B_n^{-\frac1 p}|\xi|^{-\frac n p }$, since for any $\xi\in B(0,R)$
\be \|u\|_p\ge \bigg(\int_{B(0,|\xi|)}u^\#(\eta)^pd\eta\bigg)^{\frac 1p}\ge u^\#(\xi) (B_n|\xi|^n)^{\frac 1p}.\ee
Hence,
\be\ba&\int_{B(0,R)}|\xi|^{\frac{bq}{a}} u^\#(\xi)^qd\xi =\int_{B(0,R)}|\xi|^{\frac{bq}{a}} u^\#(\xi)^{q-p}u^\#(\xi)^pd\xi 
\cr&\le \int_{B(0,R)}|\xi|^{\frac{bq}{a}} \big(\|u\|_pB_n^{-\frac1p}|\xi|^{-\frac np}\big)^{q-p}u^\#(\xi)^pd\xi \cr& =B_n^{1-\frac q p}\|u\|_p^{q-p}\int_{B(0,R)}|\xi|^{\frac{bq}{a}-\frac {nq} p+n} u^\#(\xi)^pd\xi =B_n^{-\frac q n}\|u\|_p^{q-p}\int_{B(0,R)} |\xi|^{\frac{(b-a)q}{a}}u^\#(\xi)^pd\xi\cr&\le B_n^{-\frac q n}R^{\frac{(b-a)q}{a}}\|u\|_p^q = B_n^{-\frac q n-\frac{(b-a)q}{na}}\tau^{\frac{(b-a)q}{na}}\|u\|_p^q\ea\ee
so that 
\be II\le 8BB_n^{\frac bn}B_n^{-\frac {a}{n}-\frac{(b-a)}{n}}\tau^{\frac{b-a}{n}} \|u\|_p^a=
8B\tau^{\frac{b-a}{n}} \|u\|_p^a.\ee

\smallskip
To estimate $III$, we observe that the function on $\R^n$ 
\be F(\xi)=\int_{|\xi|}^R  \rho^{b-1}\big(u^\#(\rho \xi^*)\big)^ad\rho\ee
is in $W_0^{1,p}(B(0,R))$, and $|\nabla F|=|\xi|^{b-1}(u^\#)^a$ (note that $u^\#(\rho \xi^*)$ only depends on $\rho$), hence from the Sobolev inequality \eqref{I} applied to $F$  we get 
\be III=  8BB_n^{\frac bn}\|F\|_{\frac qa}\le  8BB_n^{\frac bn}K\Big(n,\frac{qn}{an+q}\Big)\|\nabla F\|_{\frac{qn}{an+q}},\label{sobF}\ee
and
\be\ba & \|\nabla F\|_{\frac{qn}{an+q}}=\bigg(\int_{B(0,R)} \Big(|\xi|^{b-1}u^\#(\xi)^a\Big)^{\frac{qn}{an+q}}d\xi\bigg)^{\frac{an+q}{qn}}\cr
&\le\bigg(\int_{B(0,R)} u^\#(\xi)^{a\frac{qn}{an+q}\frac{an+q}{(q+n)a}}d\xi\bigg)^{\frac{an+q}{qn}\frac{(q+n)a}{an+q}}
\bigg(\int_{B(0,R)} |\xi|^{(b-1)\frac{qn}{an+q}\frac{an+q}{(1-a)q}}d\xi\bigg)^{\frac{an+q}{qn}\frac{(1-a)q}{an+q}}\cr
&=\|u^\#\|_p^a \bigg(\int_{B(0,R)} |\xi|^{-\frac{1-b}{1-a}n}d\xi\bigg)^{\frac{1-a}{n}}=B_n^{\frac{1-a}n}\Big(\frac{1-a}{b-a}\Big)^{\frac{1-a}n}R^{b-a}\|u\|_p^a\ea\ee
and finally
\be\ba  III&\le  8BB_n^{\frac bn}B_n^{\frac{1-a}n}\Big(\frac{1-a}{b-a}\Big)^{\frac{1-a}n}R^{b-a}K\Big(n,\frac{qn}{an+q}\Big)\|u\|_p^a\cr& = 8BB_n^{\frac 1n}\Big(\frac{1-a}{b-a}\Big)^{\frac{1-a}n}\tau^{\frac{b-a}n}K\Big(n,\frac{qn}{an+q}\Big)\|u\|_p^a,
\ea
\ee 
which concludes the proof.

\qed

Here is the analogue of Theorem \ref{smallprofile-sobolev} in the context of Moser-Trudinger inequalities.

\bigskip
\begin{theorem}\label{smallprofile-MT} Assume that on a  complete Riemannian manifold $I_0(0^+)=1$, and  there exist constants $\tau_0>0$, with $\tau_0<\vol(M)$, $B\ge0$, and  $0<a<1$ such that 
\be I_0(v)\ge 1- B v^{\frac an},\qquad 0<v\le \tau_0.\label{cond2}\ee
Then, for $k=1,2$ the Moser-Trudinger inequality with constant $\gamma_{n,k}$ holds for small volumes.
\end{theorem}

\ni{\bf Proof.} Let us pick $\tau\le \tau_0$ so small so that 
\be 1-B\tau^{\frac an}\ge  \frac12,\qquad   I_0(v)\le 2 \quad {\text {for}} \;\;0\le v\le \tau.\label{tt}\ee

For  $k=1$ it is enough to prove the statement for  $u\in C_c^\infty(\Omega)$, with $\vol(\Omega)\le \tau$, and such that \be \|\nabla u\|_n\le 1.\ee

As in  \eqref{repformula1}
\be u^\#(\xi)\le \int_{|\xi|}^R(1+2BB_n^{\frac an}\rho^a) I_0(B_n\rho^n) (-\partial_\rho) u^\#(\rho\xi^*)d\rho,\quad |\xi|\le R,
\label{repformula2}
\ee
which implies, for any $\xi^*\in S^{n-1}$, and $0<t<\tau$,
\be\ba u^*(t)=u^\#\big(B_n^{-\frac1n} t^{\frac1n} \xi^*\big)&\le 
 \int_t^{\tau}B_n^{-\frac1n}n^{-1}v^{-\frac{n-1}n}\big(1+2Bv^{\frac an}\big) I_0(v)|\nabla u^\#(B_n^{-\frac1n} v^{\frac1n}\xi^*)|dv\cr&:=\int_0^{\tau} k(t,v) f(v)dv
 \ea
\ee
where
\be k(t,v)=\begin{cases}B_n^{-\frac1n}n^{-1}v^{-\frac{n-1}n}\big(1+2Bv^{\frac an}\big) & {\text{if }} t<v\le \tau\cr 0 & \text{otherwise}\end{cases}
\ee
and 
\be f(v)=I_0(v)|\nabla u^\#(B_n^{-\frac1n}v^{\frac1n}\xi^*)|.\ee
It is easy to check that for $0<s\le \tau$
\be\ba\sup_{0<t\le \tau} \big(k(t,\cdot)\big)^*(s)&=\sup_{0<v\le \tau} \big(k(\cdot,v)\big)^*(s)\cr&=B_n^{-\frac1n}n^{-1} s^{-\frac{n-1}n}\big(1+2Bs^{\frac an}\big)=\gamma_{n,1}^{-\frac{n-1}n} s^{-\frac{n-1}n}\big(1+2Bs^{\frac an}\big)\ea
\ee
and also, using \eqref{ps},
\be \ba \int_0^{\tau} f(v)^ndv &=\int_0^{\tau}I_0(v)^n|\nabla u^\#(B_n^{-\frac1n}v^{\frac1n}\xi^*)|^ndv\cr&=\int_{B(0,R)}I_0(B_n|\xi|^n)^n |\nabla u^\#(\xi)|^nd\xi\le  \int_M |\nabla u(x)|^n d\mu(x)\le 1\ea\ee

The desired result follows now from  the general Theorem 1 in \cite{fm}, applied  in the context of the 1-dimensional measure space $([0,\tau], m)$, to the operator 
\be Tf(t)=\int_0^{\tau} k(t,v) f(v)dv.\ee

Now let $k=2<n$,  $u\in W_0^{1,2}(\Omega)\cap W^{2,n/2}(\Omega)$, with $\vol(\Omega)=\tau\le\tau_0<\vol(M)$, satisfying \eqref{tt}, and assume
\be \|\Delta u\|_{n/2}\le 1.\ee 
If $f=\Delta u\in L^{n/2}(\Omega)\subseteq L^{1}(\Omega)$, then 
using Talenti's comparison estimate  \eqref{estlap}, \eqref{cond2},  \eqref{tt}, Fubini's theorem, and some elementary estimates, we obtain
\be\ba u^*(t)&\le  n^{-2}B_n^{-\frac2 n}\int_t^{\tau} v^{-2+\frac 2n}(1+2B v^{\frac an}) \int_0^v f^*(w)dw \;dv\cr&\le\frac{B_n^{-\frac2 n}}{n(n-2)}t^{-1+\frac2n}(1+C t^{\frac an})
\int_0^t f^*(w)dw\cr&\qquad + \frac{B_n^{-\frac2 n}}{n(n-2)}\int_t^\tau f^*(w)w^{-1+\frac2n}(1+Cw^{\frac an})dw
\ea
\ee
for some $C>0$. Noting that $\gamma_{n,2}=n(n-2)^{\frac n{n-2}}\omega_{n-1}^{\frac2{n-2}}=\big(n(n-2)\big)^{\frac n{n-2}}B_n^{\frac2{n-2}}$, and letting 
\be k(t,v)=\begin{cases}\gamma_{n,2}^{-\frac{n-2}n}v^{-\frac{n-2}n}\big(1+Cv^{\frac an}\big) & {\text{if }} t<v\le \tau\cr 0 & \text{otherwise}\end{cases}
\ee
we get 
\be u^*(t)\le \int_0^\tau k(t,v)f^*(v)dv,\qquad 0<t< \tau,\ee
so  we can once again apply Theorem 1 in \cite{fm} to conclude.

\qed

\vskip1em


\section{From small volumes to global inequalities}\label{s6}

\smallskip
\begin{theorem}\label{loc-glob1} Given $p,q$ as in \eqref{par}, $0<\a\le 1$,  and $A>0$, if the Sobolev inequality $I_p^\a$ with constant $A$ holds for small volumes on $M$, then it holds globally on $M.$ In particular, if there is $A,\tau>0$, $B\ge0$ such that for any open set $\Omega$ with $\vol(\Omega)\le \tau$, and for all $u\in W^{1,p}_0(\Omega)$
\be \|u\|_q^\a\le A\|\nabla u\|_p^\a+B\|u\|_p^\a, \ee
then, for all $u\in W^{1,p}(M)$
 \be \|u\|_q^\a\le A\|\nabla u\|_p^\a+(B+2\tau^{-\frac\a n})\|u\|_p^\a. \ee

\end{theorem}

\ni{\bf Proof.}  To prove the global inequalities it is enough to assume  $u\in C_c^\infty(M)$. If  $u=u^+-u^-$, then $u^+=u\chi_{\{u>0\}}\in W_0^{1,p}(\{u>0\})$ and $\nabla u^+=\nabla u\chi_{\{u>0\}}$, and similar identities hold for $u^{-}$. Using the elementary inequality 
\be \bigg(\sum_{j=1}^N a_j^q\bigg)^{\frac 1q}\le \bigg(\sum_{j=1}^N a_j^p\bigg)^{\frac 1p},\qquad q>p,\qquad a_j\ge0,\qquad N\in\N\label{ineq}\ee
it is easy to see that without loss of generality we can assume  $u\ge0.$ Indeed, if that is the case then
\be\ba \|u\|_q^\a&=\big(\|u^+\|_q^q+\|u^-\|_q^q\big)^{\frac \a q}\le\Big(\big(A\|\nabla u^+\|_p^\a+B\|u^+\|_p^\a\big)^{\frac q\a}+\big(A\|\nabla u^-\|_p^\a+B\|u^-\|_p^\a\big)^{\frac q\a}\Big)^{\frac \a q}\cr&
\le\Big(A^{\frac q\a}\|\nabla u^+\|_p^q+A^{\frac q\a}\|\nabla u^-\|_p^q\Big)^{\frac \a q}+\Big(B^{\frac q\a}\| u^+\|_p^q+B^{\frac q\a}\|u^-\|_p^q\Big)^{\frac \a q}\cr&
\le \Big(A^{\frac p\a}\|\nabla u^+\|_p^p+A^{\frac p\a}\|\nabla u^-\|_p^p\Big)^{\frac \a p}+\Big(B^{\frac p\a}\| u^+\|_p^p+B^{\frac p\a}\|u^-\|_p^p\Big)^{\frac \a p}=A\|\nabla u\|_p^\a+B\|u\|_p^\a.
\ea\ee
Assume then $u\ge0$, and that the inequality  holds on $W_0^{1,p}(\Omega)$, for any open $\Omega$ with $\vol(\Omega)\le \tau.$ If $\mu(\{u>0\})\le \tau$ there is nothing to prove, so we can assume $\mu(\{u>0\})>\tau$, i.e. $u^*(\tau)>0$. Write
\be u=\big(u-u^*(\tau)\big)^++u^*(\tau)\chi_{\{u\ge u^*(\tau)\}}+u\chi_{\{u<u^*(\tau)\}}.\ee
and using Minkowski's inequality we get
\be\ba \|u\|_q^\a&\le \bigg( \|\big(u-u^*(\tau)\big)^+\|_q+u^*(\tau)\mu\big(\{u\ge u^*(\tau)\}\big)^{\frac1q}+\bigg(\int_{\{u<u^*(\tau)\}} u^q d\mu\bigg)^{\frac1q}\bigg)^\a\cr
&\le \|\big(u-u^*(\tau)\big)^+\|_q^\a+\bigg(u^*(\tau)\mu\big(\{u\ge u^*(\tau)\}\big)^{\frac1q}\bigg)^\a+\bigg(\int_{\{u<u^*(\tau)\}} u^q d\mu\bigg)^{\frac \a q}.
\ea\ee

From \eqref{tau} and the fact that $q>p$ we have
\be\ba u^*(\tau)\mu\big(\{u\ge u^*(\tau)\}\big)^{\frac1q}&\le \tau^{\frac1q-\frac1p}u^*(\tau)\mu\big(\{u\ge u^*(\tau)\}\big)^{\frac1p}\cr&=\tau^{-\frac1n}\bigg(\int_{\{u\ge u^*(\tau)\}} u^*(\tau)^p d\mu\bigg)^{\frac1p}\le \tau^{-\frac1n}\|u\|_p.\ea\ee

From \eqref{set1} and \eqref{int1} we obtain
\be\int_{\{u<u^*(\tau)\}} u^q d\mu=\int_M u^q d\mu -\int_{\{u\ge u^*(\tau)\}} u^q d\mu\le \int_\tau^\infty \big(u^*(t)\big)^qdt\ee
and  therefore, using \eqref{ineq}
\be\ba\bigg(\int_{\{u<u^*(\tau)\}} u^q d\mu\bigg)^{\frac1q}&\le \bigg(\int_\tau^\infty \big(u^*(t)\big)^qdt\bigg)^{\frac1q}=
\bigg(\sum_{k=1}^\infty\int_{k\tau}^{(k+1)\tau} \big(u^*(t)\big)^qdt\bigg)^{\frac1q}\cr&\le
\bigg(\sum_{k=1}^\infty\tau\big(u^*(k\tau)\big)^q\bigg)^{\frac1q}\le \tau^{\frac1q-\frac1p}\bigg(\sum_{k=1}^\infty\tau\big(u^*(k\tau)\big)^p\bigg)^{\frac1p}\cr&\le  \tau^{\frac1q-\frac1p}\bigg(\sum_{k=1}^\infty\int_{(k-1)\tau}^{k\tau} \big(u^*(t)\big)^pdt\bigg)^{\frac1p}=
\tau^{-\frac1n}\|u\|_p.
\ea\ee
Since $\big(u-u^*(\tau)\big)^+\in W_0^{1,p}(\{u-u^*(\tau)>0\})$ and $\mu\big(\{u-u^*(\tau)>0\}\big)\le \tau$, under the given assumption of  Sobolev inequality for small volumes we obtain
\be\ba \|u\|_q^\a&\le A\|\nabla \big(u-u^*(\tau)\big)^+\|_p^\a+B\|\big(u-u^*(\tau)\big)^+\|_p^\a+2\tau^{-\frac \a n}\|u\|_p^\a
\cr&\le A\|\nabla u\|_p^\a+(B+2\tau^{-\frac \a n})\|u\|_p^\a.\ea
\ee
\rightline\qed

\smallskip
\begin{theorem}\label{loc-glob2} Given $k=1,2$ and $\gam>0$,   if the Moser-Trudinger inequality with constant $\gam$ holds for small volumes on $M$, then it holds globally on $M$.

\end{theorem}

\ni {\bf Proof.}
It is enough to assume that $u\in C_c^\infty(M)$, satisfying \eqref{ruf}, i.e.
\be\kappa\|u\|_{n/k}^{n/k}+\|\nabla^k u\|_{n/k}^{n/k}\le 1.\ee
 
 It's easy to check that for any such $u$ we have 
 \be \bigg| \int_M \exp_{\llceil\frac{n}{k}-2\rrceil}{\left(\gam |u|^{\frac{n}{n-k}}\right)}d\mu- \int_{\{|u|\ge1\}}e^{\sgam  |u|^{ \frac n{n-k}}}d \mu\bigg|\le e^{\sgam}\|u\|_{n/k}^{n/k}\label{expreg}\ee
(see e.g. Lemma 9 in \cite{fm}) so that  it is enough to prove that 
\be \int_{\{|u|\ge1\}}e^{\sgam  |u|^{\frac n{n-k}}}d\mu\le C.\label{MT1}\ee
Note that under our hypothesis $\mu(\{|u|\ge1\})\le \|u\|_{n/k}^{n/k}\le 1$.

Note now that that 
\be \mu( \{x:|u(x)|>u^*(\tau)\})\le \tau\le \mu(\{x: |u(x)|\ge u^*(\tau)\})\label{tau1}\ee
and that $u^*(\tau)>0$ if and only if $\mu(\{|u|>0\})>\tau$.

If $\mu(\{|u|>0\})\le \tau$ then \eqref{MTL} follows from \eqref{MT1} with $\Omega=\{|u|>0\})$ and the  given  assumptions.

 Let us assume that $\mu(\{|u|>0\})>\tau$,  and let 
\be v=(u^+-u^*(\tau))^+-(u^--u^*(\tau))^+=\bc 

u-u^*(\tau) & {\text {if}} \; u^+>u^*(\tau)\cr
 u+u^*(\tau) & {\text {if}} \; u^-> u^*(\tau)\cr
0& {\text {if}} \; u^+\le u^*(\tau),\;  u^-\le u^*(\tau).\ec\ee
Letting
\be\Omega_\tau=\{|u|>u^*(\tau)\}\ee
then we have $v\in W_0^{1,n}(\Omega_\tau)$ for $k=1$, and $v\in W_0^{1,2}(\Omega_\tau)\cap W^{2,\frac n2}(\Omega_\tau)$, for $k=2$. Moreover,   $\mu(\Omega_\tau)\le \tau$, and
\be \|\nabla^k v\|_{ n/k}^{n/k}=\int_{\Omega_\tau}|\nabla^k u|^{\frac nk}d\mu\le \|\nabla^k u\|_{n/k}^{n/k}\le 1-\kappa \|u\|_\nk^\nk,\label {v100}\ee
\be |u|\le|v|+u^*(\tau).\ee
Using the inequality $(a+b)^p\le a^p \e^{1-p}+b^p (1-\e)^{1-p}$ ($a,b\ge0,p\ge1, 0<\e<1$) we then have
\be e^{\sgam  u^\nnk}\le  e^{\sgam  v^\nnk \e^{-\knk} }\,e^{\sgam  u^*(\tau)^\nnk  (1-\e)^{-\knk}}.\label{exp}
 \ee

From \eqref{tau1}
 we obtain 
 \be \ba0<u^*(\tau)^{\frac nk}&\le u^*(\tau)^{\frac nk}\;\frac{\mu(\{|u|\ge u^*(\tau)\})}\tau \le\frac1\tau\; \|u\|_\nk^\nk \le\frac1{\kappa\tau}\;( 1- \|\nabla^ku\|_\nk^\nk),\label{u3}\ea\ee
 which implies $\|\nabla^k u\|_\nk<1$.

 If  $\|\nabla^k u\|_{n/k}=0$, then $u=0$, so \eqref{MT1} follows.

If $\|\nabla^k u\|_\nk>0$, let $\e=\|\nabla^k u\|_\nk^\nk$, in which case 
using the inequality for small volumes  and \eqref{v100}, \eqref{u3}
\be \int_{\Omega_\tau}e^{\sgam  v^\nnk \e^{-\knk} }\le C,\qquad e^{\sgam  u^*(\tau)^\nnk  (1-\e)^{-\knk}}\le e^{\sgam (\kappa\tau)^{-\knk}}\label{u2}.\ee
Then  \eqref{MT1} follows, since $\{u\ge1\}\subseteq \Omega_\tau\cup \{1\le u\le u^*(\tau)\}\subseteq \Omega_\tau\cup \{1\le u\le(\kappa\tau)^{-k/n}\}$, and    $\mu(\{u\ge1\})\le1.$

\rightline\qed


\vskip1em


\section{Proofs of Theorems \ref{main1} and \ref{main2}}\label{s7}

\bigskip
The proofs of Theorems \ref{main1} and \ref{main2}  follow immediately from Theorems \ref{expansion}, \ref{smallprofile-sobolev}, \ref{smallprofile-MT}, \ref{loc-glob1}, \ref{loc-glob2}. 

\qed

\section{Appendix}\label{appendix}

This is an outline of the proof of Proposition \ref{talenti-man}. 

For $u\in W_0^{1,2}(\Omega)$ the distributional Laplacian $\Delta u$ is in $L^1(\Omega)$ if there is $f\in L^1(\Omega)$ such that 
\be \int_\Omega \langle\nabla u,\nabla\varphi\rangle d\mu =\int_\Omega f\varphi\; d\mu\label{identity}\ee
for all $\varphi\in C_c^\infty(\Omega)$. We set $f=-\Delta u$.
It is routine to show (by smooth truncation of approximating $C_c^\infty$ functions and dominated convergence, since $f\in L^1$) to show that \eqref{identity} is valid for all $\varphi\in W_0^{1,2}\cap L^\infty(\Omega)$.

Let
\be \mu(s)=\mu(\{|u|>s\})\ee
and
\be\Phi(s)=\int_{|u|>s} |\nabla u|^2 d\mu.\ee
By applying  \eqref{identity} to $\varphi= \min\big\{(|u|-s)^+,h\big\}{\rm sgn} \; u\in W_0^{1,2}\cap L^\infty(\Omega)$ we  get, for $s\ge0$ and $h>0$,
\be\ba \Phi(s)-\Phi(s+h)=\int_{s<|u|\le s+h} |\nabla u|^2 d\mu&=\int_\Omega f\varphi
\le h\int_{|u|>s}|f|d\mu\cr& \le h \int_0^{\mu(s)} f^* dw\le h\|f\|_1\ea\ee
which shows that $\Phi$ is nonincreasing and Lipschitz on $[0,\infty)$, and  moreover
\be -\Phi'(s)\le \int_0^{\mu(s)} f^*(w)dw,\label{idphi}\ee
for a.e. $s>0$.

The Fleming-Rishel's identity gives for $0\le s< s'\le\infty$

\be \int_{s<|u|\le s'} |\nabla u|d\mu=\int_s^{s'}\per(\{|u|>\sigma\})d\sigma \label{fr0}\ee

 then we have for a.e. $s>0$ 
\be -\frac d{ds} \int_{|u|>s} |\nabla u|d\mu=\per(\{|u|>s\})\ge n B_n^{\frac 1n}\mu(s)^{\frac {n-1}n} I_0(\mu(s))\label{fr1}\ee
which follows from \eqref{fr0}, and the definition of $I_0(v)$.
Additionally, using Cauchy-Schwarz's inequality on $\{s<|u|\le s+h\}$ and \eqref{idphi}, for a.e. $s\in(0,\|u\|_\infty)$
\be\ba 0<\Big(\per(\{|u|>s\})\Big)^2&=\bigg(-\frac d{ds} \int_{|u|>s} |\nabla u|d\mu\bigg)^2\cr&\le \big(-\mu'(s)\big)\big(-\Phi'(s)\big)\le-\mu'(s)\int_0^{\mu(s)} f^* dw<\infty,\label{fr2}\ea\ee
where the first inequality is justified by the fact that $0<\mu(\{|u|>s\})\le \mu(\Omega)<\mu(M)$, so the perimeter of $\{|u|>s\}$ cannot vanish  (more generally: for any $E\subseteq M$ measurable, if $\per(E)=0$ then either $\mu(E)=0$ or $\mu(M\setminus E)=0$, which  in turn follows from the fact that $\nabla \chi_E=0$ and $M$ is connected). The last inequality follows since $\mu$ is decreasing, and $f\in L^1(\Omega)$.

Hence, combining \eqref{fr1} and \eqref{fr2}, 
\be 1\le n^{-2}B_n^{-\frac2n} I_0(\mu(s))^{-2}\mu(s)^{-2\frac{n-1}n}\bigg(\int_0^{\mu(s)} f^* dw\bigg) (-\mu'(s)),\label{fr3}\ee
for a.e. $s\in (0,\|u\|_\infty)$, where the right-hand side is $+\infty$ if $I_0(\mu(s))=0$.

Integrating in $s$, using the fact that $-\mu(s)$ is increasing, we obtain
\be s\le n^{-2}B_n^{-\frac2n}\int_{\mu(s)}^{\vol(\Omega)} v^{-2\frac{n-1}n} I_0(v)^{-2}\int_0^v f^* dw dv,\qquad 0<s<\|u\|_\infty,\label{fr4}\ee
where we used the inequality $\int_a^b H(\phi)(\phi')\,d\sigma\le\int_{\phi(a)}^{\phi(b)}H\,dv$, with $\phi$ increasing, $0\le H\le+\infty $ Borel, which in turn follows from the classical Lebesgue’s inequality $\int_\alpha^\beta\phi'\le\phi(\beta)-\phi(\alpha)$ for monotone $\phi$. 

Estimate \eqref{fr4}  implies \eqref{estlap}, indeed, if $0<t<\mathrm{Vol}(\Omega)$ and $0<s<u^*(t)$, then $\mu(s)>t$ by (2.1), so (8.10) gives $s\le n^{-2}B_n^{-\frac2n}\int_t^{\mathrm{Vol}(\Omega)}(\cdots)\,dv$; let $s\uparrow u^*(t)$.

 \qed

\bigskip

\begin{rk} In the classical theorem in \cite{talenti1} on $\R^n$, the function $f=-\Delta u$ is in  either  $L^{\frac{2n}{n+2}}(\Omega)$, when $n\ge3$, or in $L^p(\Omega)$ for any  $p>1$, when $n=2$, and identity \eqref{identity} is valid for any $\varphi\in W_0^{1,2}(\Omega)$. This is a consequence of the Sobolev embedding $W_0^{1,2}(\Omega)\hookrightarrow L^{\frac{2n}{n-2}}(\Omega)$ (or $W_0^{1,2}(\Omega)\hookrightarrow L^q(\Omega)$ for any $q\in [1,\infty)$ when $n=2$). 
\end{rk}
\begin{rk} The hypothesis $\mathrm{Vol}(\Omega)<\mathrm{Vol}(M)$ cannot be dropped. Indeed, if $\mathrm{Vol}(M)<~\infty$, take $\Omega=M$ and  $u\equiv1$. Then $u\in W^{1,2}(M)=W^{1,2}_0(M)$, and $\Delta u=0$. Thus the right-hand side of
\eqref{estlap} vanishes, but $u^*=1$ on $(0,\mathrm{Vol}(M))$. 
\end{rk}

\vskip1em
\begin{table}[h]

\setlength{\tabcolsep}{0pt} 
\begin{tabular}{@{}p{0.5\linewidth}p{0.5\linewidth}@{}}
\textbf{Carlo Morpurgo} & \textbf{Liuyu Qin}  \\
Department of Mathematics & Department of Mathematics and Statistics\\
University of Missouri & Hunan University of Finance and Economics\\
Columbia, Missouri 65211 & Changsha, Hunan  \\
USA & China\\
\texttt{morpurgoc@umsystem.edu} & \texttt{Liuyu\_Qin@outlook.com} \\

\end{tabular}
\end{table}
\end{document}